\documentclass[11pt, reqno]{article}
\usepackage[normalem]{ulem}

\title
{\textsc{ Uniform Spanning Forests of Planar Graphs}}
\author{Tom Hutchcroft and Asaf Nachmias}
\date{\today}

\usepackage{amsmath,amssymb,amsfonts,amsthm}
\usepackage{color}

\usepackage{graphicx}

\usepackage{esvect}

\usepackage{bbm}

\usepackage{caption}
\usepackage{thmtools}
\usepackage{commath}
\usepackage{mathrsfs}

\usepackage{enumitem}

\usepackage{mathtools}
\usepackage{graphicx}
\usepackage[colorlinks=true,linkcolor=blue,citecolor=blue,urlcolor=blue,pdfborder={0 0 0}]{hyperref}
\usepackage{cleveref}
\usepackage[numeric,initials,nobysame]{amsrefs}

\crefname{theorem}{Theorem}{Theorems}
\crefname{thm}{Theorem}{Theorems}
\crefname{lemma}{Lemma}{Lemmas}
\crefname{lem}{Lemma}{Lemmas}
\crefname{remark}{Remark}{Remarks}
\crefname{prop}{Proposition}{Propositions}
\crefname{defn}{Definition}{Definitions}
\crefname{corollary}{Corollary}{Corollaries}
\crefname{conjecture}{Conjecture}{Conjectures}
\crefname{question}{Question}{Questions}
\crefname{chapter}{Chapter}{Chapters}
\crefname{section}{Section}{Sections}
\crefname{figure}{Figure}{Figures}
\newcommand{\crefrangeconjunction}{--}

\theoremstyle{plain}
\newtheorem{thm}{Theorem}[section]
\newtheorem{lemma}[thm]{Lemma}
\newtheorem{lem}[thm]{Lemma}
\newtheorem{corollary}[thm]{Corollary}
\newtheorem{prop}[thm]{Proposition}

\newtheorem{question}[thm]{Question}
\theoremstyle{definition}

\theoremstyle{remark}
\newtheorem{remark}[thm]{Remark}

\numberwithin{equation}{section}

\renewcommand{\P}{\mathbb P}
\newcommand{\E}{\mathbb E}
\newcommand{\C}{\mathbb C}
\newcommand{\R}{\mathbb R}
\newcommand{\Z}{\mathbb Z}
\newcommand{\N}{\mathbb N}

\newcommand{\F}{\mathfrak F}

\newcommand{\cE}{\mathcal E}
\newcommand{\cF}{\mathcal F}

\newcommand{\cL}{\mathcal L}

\newcommand{\cR}{\mathcal R}

\newcommand{\sA}{\mathscr A}
\newcommand{\sB}{\mathscr B}
\newcommand{\sC}{\mathscr C}
\newcommand{\sD}{\mathscr D}

\newcommand{\sR}{\mathscr R}

\newcommand{\bbG}{\mathbb G}
\newcommand{\bbH}{\mathbb H}

\newcommand{\fA}{\mathfrak A}

\newcommand{\fC}{\mathfrak C}

\newcommand{\fF}{\mathfrak F}

\newcommand{\fH}{\mathfrak H}

\newcommand{\WUSF}{\mathsf{WUSF}}
\newcommand{\FUSF}{\mathsf{FUSF}}

\newcommand{\UST}{\mathsf{UST}}

\numberwithin{equation}{section}

\newcommand{\D}{\mathbb{D}}
\newcommand{\ReffW}{\mathscr{R}_{\mathrm{eff}}^\mathrm{W}}
\newcommand{\ReffF}{\mathscr{R}_{\mathrm{eff}}^\mathrm{F}}
\newcommand{\Reff}{\mathscr{R}_{\mathrm{eff}}}

\newcommand{\CeffF}{\mathscr{C}_{\mathrm{eff}}^\mathrm{F}}
\newcommand{\Ceff}{\mathscr{C}_{\mathrm{eff}}}

\newcommand{\eps}{\varepsilon}
\newcommand{\wusf}{\WUSF}

\newcommand{\bP}{\mathbf{P}}
\newcommand{\bE}{\mathbf{E}}
\newcommand{\past}{\mathrm{past}}

 \newcommand{\barraydebug}{\begin{array}}
 \newcommand{\earraydebug}{\end{array}}
\newcommand{\ah}{a_\bbH}
\newcommand{\diam}{\mathrm{diam}}
\newcommand{\area}{\mathrm{area}}

\newcommand{\Pp}[1]{P(#1)}
\newcommand{\Pd}[1]{P^\dagger(#1)}
\newcommand{\bM}{\mathbf M}

\newcommand{\bigmid}{\;\ifnum\currentgrouptype=16 \middle\fi|\;}

\DeclareMathOperator{\sgn}{sgn}
\usepackage[margin=1in]{geometry}
\newcommand{\tom}[1]{{#1}}

\newcommand{\tomm}[1]{{#1}}

\newcommand{\asaf}[1]{{#1}}

\begin{document}

\maketitle

\begin{abstract}
We prove that the free uniform spanning forest of any bounded degree proper plane graph
is connected almost surely, answering a question of Benjamini, Lyons, Peres and Schramm.

We provide a quantitative form of this result,
calculating the critical exponents governing the geometry of the uniform spanning forests of transient proper plane graphs with bounded degrees and codegrees. We find that the same exponents hold universally over this entire class of graphs provided that measurements are made using the hyperbolic geometry of their circle packings rather than their usual combinatorial geometry.
\end{abstract}

%%%%%%%%%%%%%%%%%%%%%%%%%%%%%%%%%%%%%%%%%%%%%%%%%%%%%%%%%%%%%%%%%%
% [DRAFT]
\section{Introduction}

The \textbf{uniform spanning forests} (USFs) of an infinite, locally finite, connected graph $G$ are defined as weak limits of uniform spanning trees of finite subgraphs of $G$. These limits can be taken with either \textbf{free} or \textbf{wired} boundary conditions, yielding the \textbf{free uniform spanning forest (FUSF)} and the \textbf{wired uniform spanning forest (WUSF)} respectively.
% USFs are closely connected to several other areas of probability theory
Although the USFs are defined as limits of random spanning trees, they need not be connected. Indeed, a principal result of Pemantle \cite{Pem91} is that the WUSF and FUSF of $\Z^d$ coincide, and that they are almost surely (a.s.) a single tree if and only if $d \leq 4$.
Benjamini, Lyons, Peres and Schramm \cite{BLPS} (henceforth referred to as BLPS) later gave a complete characterisation of connectivity of the WUSF, proving that the WUSF of a graph is connected a.s.\ if and only if two independent random walks on the graph intersect a.s.\ \cite[Theorem 9.2]{BLPS}.

%In particular, the
% Besides this, exhibited connections to simple random walks, potential theory, amenability and much more.
% In particular, they showed that the two measures differ if and only if the graph admits harmonic functions of finite Dirichlet energy.

% \begin{wrapfigure}{r}{0.45\textwidth}
% \includegraphics[width=0.45\textwidth]{halfandhalf}
% \end{wrapfigure}
The FUSF is much less understood. No characterisation of its connectivity is known, nor has it been proven that connectivity is a zero-one event.
%it is not even known if the number of components of the FUSF of a graph is necessarily non-random.
% Indeed, it has only recently been proven that the number of components of the FUSF of any Cayley graph is $1$ or $\infty$ a.s.\ \cite{HN15a,Timar15}.
One class of graphs in which the FUSF is relatively well understood are the \emph{proper plane graphs}.
%For these graphs, duality between the FUSF and the WUSF allows us to reformulate questions about the FUSF of a proper plane graph as questions about the WUSF of the graph's plane dual.
Recall that a \textbf{planar graph} is a graph that can be embedded in the plane, while a \textbf{plane graph} is a planar graph $G$ together with a specified embedding of $G$ in the plane (or some other topological disc). A plane graph is \textbf{proper} if the embedding is proper, meaning that every compact subset of the plane (or whatever topological disc $G$ was embedded in) intersects at most finitely many edges and vertices of the drawing (see \cref{Sec:Embeddings} for further details).
%A graph is a \textbf{proper planar} graph if it admits a proper embedding into the plane.
For example, every tree can be drawn in the plane without accumulation points, while the product of $\Z$ with a finite cycle is planar but cannot be drawn in the plane without an accumulation point (and therefore has no proper embedding in the plane).
% ; rather, it can be drawn without accumulation points in the punctured plane $\C\setminus\{0\}$.

BLPS proved that the free and wired uniform spanning forests are distinct
% , from which they deduced
% One of the many results in \cite{BLPS} is that FUSF $\neq$ WUSF
whenever $G$ is a transient proper plane graph with bounded degrees, and
% and a bounded number of edges in each face.
asked \cite{BLPS}*{Question 15.2} whether the FUSF is a.s.\ connected in this class of graphs.  They proved that this is indeed the case when $G$ is a self-dual plane Cayley graph that is rough-isometric to the hyperbolic plane \cite{BLPS}*{Theorem 12.7}. These hypotheses were later weakened by Lyons, Morris and Schramm \cite[Theorem 7.5]{LMS08}, who proved  that the FUSF of any bounded degree proper plane graph that is rough-isometric to the hyperbolic plane is a.s.~connected.

Our first result provides a complete answer to \cite{BLPS}*{Question 15.2}, obtaining optimal hypotheses under which the FUSF of a proper plane graph is a.s.\ connected. The techniques we developed to answer this question also allow us to prove quantitative versions of this result, which we describe in the next section.
% This result is only the
 We state our result in the natural generality of proper plane networks. Recall that a \textbf{network} $(G,c)$ is a locally finite, connected graph $G=(V,E)$ together with a function $c:E\to(0,\infty)$ assigning a positive \textbf{conductance} to each edge of $G$. The \textbf{resistance} of an edge $e$ in a network $(G,c)$ is defined to be $1/c(e)$. The uniform spanning forest of a finite network gives each tree probability proportional to the product of the conductances of its edges. The free and wired USFs of an infinite network are obtained, as before, by taking weak limits over exhaustions.  Graphs may be considered as networks by setting $c\equiv1$. A \textbf{plane network} is a planar graph $G$ together with specified conductances and a specified drawing of $G$ in the plane.

\begin{thm}
  \label{T:mainthm}
	The free uniform spanning forest is almost surely connected in any bounded degree proper plane network with edge conductances bounded above.
\end{thm}

In light of the duality between the free and wired uniform spanning forests of proper plane graphs (see \cref{Sec:Duality}),
the FUSF of a proper plane graph $G$ with locally finite dual $G^\dagger$ is connected a.s.\ if and only if every component of the WUSF of $G^\dagger$ is a.s.\ \emph{one-ended}. Thus, \cref{T:mainthm} follows easily from the dual statement \cref{T:planeWUSF} below. (The implication is immediate when the dual graph is locally finite.) Recall that an infinite graph $G=(V,E)$ is said to be \textbf{one-ended} if, for every finite set $K \subset V$, the subgraph induced by $V \setminus K$ has exactly one infinite connected component.  In particular, an infinite tree is one-ended if and only if it does not contain a simple bi-infinite path.
Components of the WUSF are known to be one-ended a.s.\ in several other classes of graphs \cite{Pem91,BLPS,LMS08,AL07,hutchcroft2015interlacements,H15}, and  are recurrent in any graph \cite{Morris03}. Recall that a plane graph is said to have \textbf{bounded codegree} if its dual has bounded degree.

\begin{thm}\label{T:planeWUSF}
	Every component of the wired uniform spanning forest is one-ended almost surely in any bounded codegree proper plane network with edge resistances bounded above.
\end{thm}

The uniform spanning trees of $\Z^2$ and other planar Euclidean  lattices are very well understood due to the deep theory of conformally invariant scaling limits. The study of the UST on $\Z^2$ led Schramm, in his seminal paper \cite{SLE}, to introduce the SLE processes, which he conjectured to describe the scaling limits of the loop-erased random walk and UST. This conjecture was subsequently proven in the celebrated work of Lawler,  Schramm, and Werner \cite{lawler2004conformal}.   Overall, Schramm's introduction of SLE  has revolutionised the understanding of statistical physics in two dimensions;  see e.g.\ \cite{Rohde11,garban2012quantum,lawler2008conformally} for guides to the extensive literature in this very active field.

Although our own setting is too general to apply this theory, we nevertheless keep conformal invariance in mind throughout this paper.  Indeed, the key to our proofs is \emph{circle packing}, a canonical method of drawing planar graphs that is closely related to conformal mapping (see e.g.\ \cite{RS87,HS93,HS98,St05,Rohde11} and references therein). For many purposes, one can pretend that the random walk on the packing is a quasiconformal image of standard planar Brownian motion:  Effective resistances, heat kernels, and harmonic measures on the graph can each be estimated in terms of the corresponding Brownian quantities \cite{ABGN14,Chelkak}.

It is natural to ask whether \cref{T:mainthm} extends to \emph{all} bounded degree planar graphs, rather than just those admitting proper embeddings into the plane. In the sequel to this paper, we provide an example, which was analyzed in collaboration with Gady Kozma, to show that the theorem does not admit such an extension. However, we also show there that the conclusion of \cref{T:mainthm} does continue to hold on any bounded degree planar network that admits a proper embedding into a domain with countably many boundary components.

\subsection{Universal USF exponents via circle packing}

A \textbf{circle packing} $P$ is a set of discs in the Riemann sphere $\C\cup\{\infty\}$ that have disjoint interiors (i.e.,~do not overlap) but can be tangent. The \textbf{tangency graph} $G=G(P)$ of a circle packing $P$ is the plane graph with the centres of the circles in $P$ as its vertices and with edges given by straight lines between the centres of tangent circles. The Koebe-Andreev-Thurston Circle Packing Theorem \cite{K36,Th78,marden1990thurston} states that every \emph{finite}, simple planar graph arises as the tangency graph of a circle packing, and that if the graph is a triangulation then its circle packing is unique up to M\"obius transformations and reflections (see \cite{BriSch93} for a combinatorial proof).
 The Circle Packing Theorem was extended to infinite plane triangulations by He and Schramm \cite{HeSc,HS93}, who proved that every  \emph{infinite, proper}, simple plane triangulation admits a locally finite circle packing in either the Euclidean plane or the hyperbolic plane (identified with the interior of the unit disc), but not both. (Recall that in the Poincar\'e disc model, Euclidean and hyperbolic circles in the disc coincide as sets but may have different centres and radii.) We call an infinite, simple, proper plane triangulation \textbf{CP parabolic} if it admits a circle packing in the plane and \textbf{CP hyperbolic} otherwise.

He and Schramm \cite{HeSc} also initiated the use of circle packing to study probabilistic questions on plane graphs. In particular, they showed that a bounded degree, simple, proper plane triangulation is CP parabolic if and only if simple random walk on the triangulation is recurrent (i.e., visits every vertex infinitely often a.s.).
Circle packing has since proven instrumental in the study of planar graphs, and random walks on planar graphs in particular. Most relevantly to us, circle packing was used by Benjamini and Schramm \cite{BS96a} to prove that every transient, bounded degree planar graph admits non-constant harmonic Dirichlet functions; BLPS \cite{BLPS} later applied this result to deduce that the free and wired uniform spanning forest of a bounded degree plane graph coincide if and only if the graph is recurrent.
  We refer the reader to \cite{St05} and \cite{Rohde11} for background on circle packing, and to \cite{HeSc,BS96a,BeSc,JS00,GGN13,ABGN14,Chelkak,AHNR15,GGNS15,hutchcroft2015boundaries,angel2016half,1707.07751} for further probabilistic applications.

A guiding principle of the works mentioned above is that circle packing endows a triangulation with a geometry that, for many purposes, is better than the usual graph metric.
The results described in this section provide a compelling instance of this principle in action: we find that the critical exponents governing the geometry of the USFs are universal over all transient, bounded degree, proper plane triangulations, provided that measurements are made using the hyperbolic geometry of their circle packings rather than the usual combinatorial geometry of the graphs. It is crucial here that we use the circle packing to take our measurements: The exponents in the graph distance are not universal and need not even exist (see \cref{fig:layers}). In \cref{remark:parabolic} we give an example to show that no such universal exponents hold  in the CP parabolic case at this level of generality.
 
\begin{figure}[t]
\centering
\includegraphics[height=0.375\textwidth]{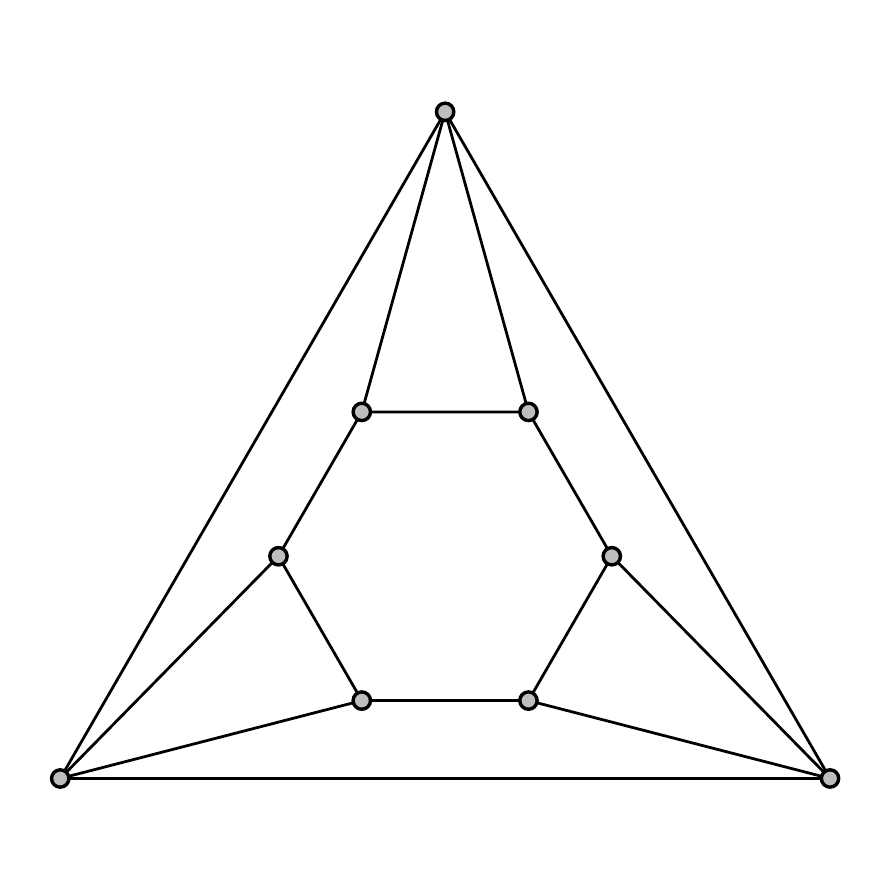} \hspace{1.1cm} \includegraphics[height=0.375\textwidth]{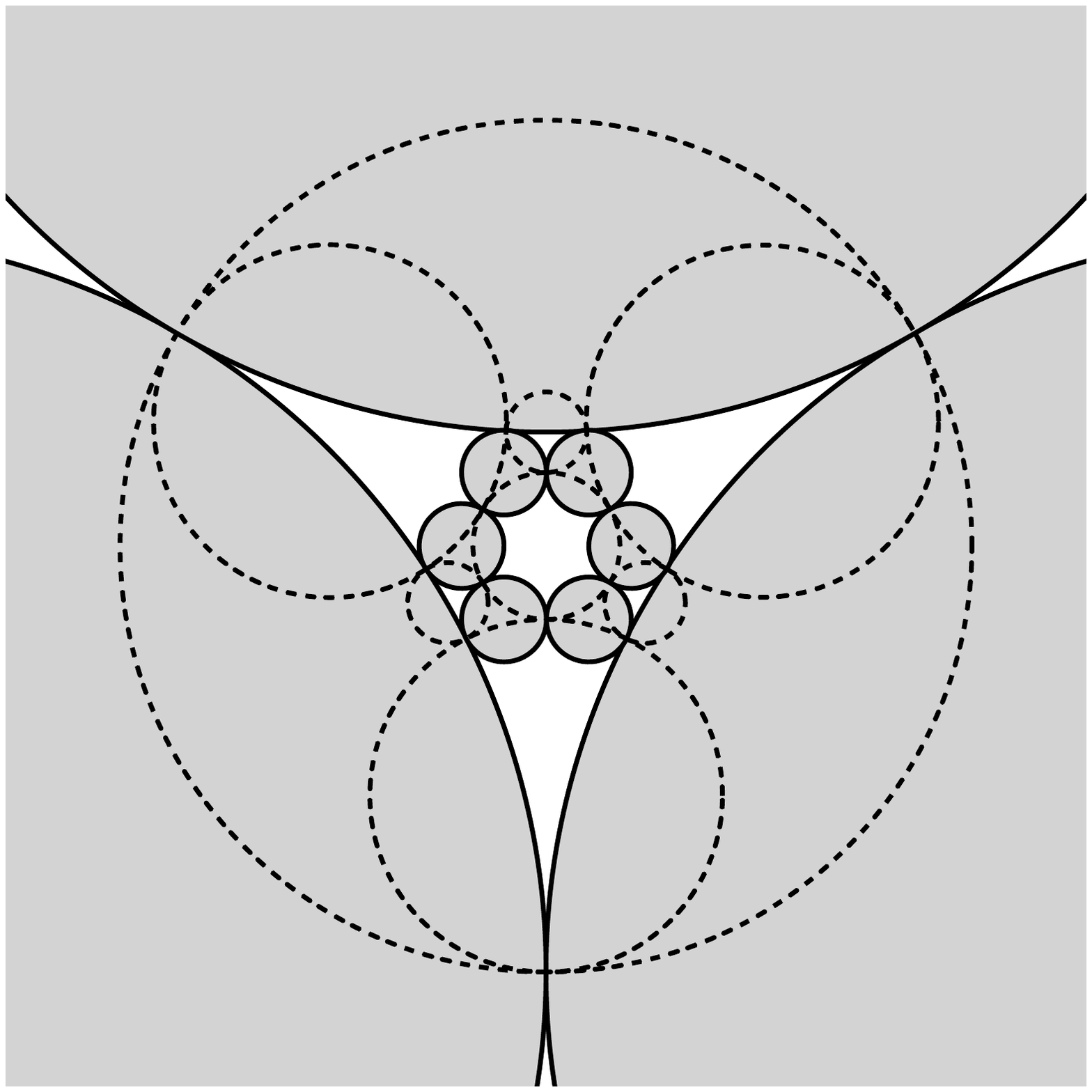}
\caption{
A finite polyhedral plane graph (left) and its double circle packing (right). Primal circles are filled and have solid boundaries, dual circles have dashed boundaries.}
\label{fig.dcp}
\end{figure}

In order to state the quantitative versions of \cref{T:mainthm,T:planeWUSF} in their full generality, we first introduce 
\textbf{double circle packing}.
Let $G$ be a plane graph with vertex set $V$ and face set $F$. A double circle packing of $G$ is a pair of circle packings $P=\{\Pp{v} : v \in V\}$ and $P^\dagger=\{\Pd{f}:f\in F\}$ in $\C \cup\{\infty\}$ satisfying the following conditions (see \cref{fig.dcp}):
\begin{enumerate}
\item (\textbf{$G$ is the tangency graph of $P$}.) For each pair of vertices $u$ and $v$ of $G$, the discs $\Pp{u}$ and $\Pp{v}$ are tangent if and only if $u$ and $v$ are adjacent in $G$.
\item (\textbf{$G^\dagger$ is the tangency graph of $P^\dagger$}.)  For each pair of faces $f$ and $g$ of $G$, the discs $\Pd{f}$ and $\Pd{g}$ are tangent if and only if $f$ and $g$ are adjacent in $G^\dagger$.
\item (\textbf{Primal and dual circles are perpendicular}.) For each vertex $v$ and face $f$ of $G$, the discs $\Pd{f}$ and $\Pp{v}$ have non-empty intersection if and only if $f$ is incident to $v$, and in this case the boundary circles of $\Pd{f}$ and $\Pp{v}$ intersect at right angles.
\end{enumerate}
% The existence
 It is easily seen that any finite plane graph admitting a double circle packing must be \textbf{polyhedral}, meaning that is it both simple
(i.e., not containing any loops or multiple edges)
  and $3$-connected (meaning that the subgraph induced by $V\setminus\{u,v\}$ is connected for each $u,v \in V$). Conversely, Thurston's interpretation of Andreev's Theorem \cite{Th78,marden1990thurston}, or the Brightwell-Scheinerman Theorem \cite{BriSch93}, implies that every finite, polyhedral plane graph admits a double circle packing. The corresponding infinite theory was developed by He \cite{he1999rigidity}, who proved that every infinite, polyhedral, proper plane graph $G$ with locally finite dual admits a double circle packing\footnote{He phrased his results more generally than this, see \cref{Sec:CP} for an explanation of how the statement given here follows from the one in \cite{he1999rigidity}.} in either the Euclidean plane or the hyperbolic plane, but not both, and that this packing is unique up to M\"obius transformations and reflections. In particular, in the hyperbolic case, the packing is unique up to isometries of the hyperbolic plane. As before, we say that $G$ is CP parabolic or CP hyperbolic as appropriate. As for triangulations \cite{HeSc}, CP hyperbolicity is equivalent to transience for graphs with bounded degrees and codegrees~\cite{he1999rigidity}.

Let $G$ be  CP hyperbolic and let $(P,P^\dagger)$ be a double circle packing of $G$ in the hyperbolic plane. We write $r_\bbH(v)$ for the hyperbolic radius of the circle $\Pp{v}$. For each subset $A \subset V(G)$, we define $\diam_\bbH(A)$  to be the hyperbolic diameter of the set of hyperbolic centres of the circles in $P$ corresponding to vertices in $A$, and define $\area_\bbH(A)$ to be the hyperbolic area of the union of the circles in $P$ corresponding to vertices in $A$. Since $(P, P^\dagger)$ is unique up to isometries of the hyperbolic plane, $r_\bbH(v)$, $\diam_\bbH(A)$, and $\area_\bbH(A)$ do not depend on the choice of $(P, P^\dagger)$.

We say that a network $G$ has \textbf{bounded local geometry} if there exists a constant $M$ such that $\deg(v) \leq M$ for every vertex $v$ of $G$ and $M^{-1}\leq c(e)\leq M$ for every edge $e$ of $G$.
Given a plane network $G$, we let $F$ be the set of faces of $G$, and define
\[\bM=\bM_G=\asaf{\max}\Big\{\sup_{v\in V} \deg(v),\; \sup_{f\in F} \deg(f),\;
\; \sup_{e\in E} c(e),\; \sup_{e\in E}c(e)^{-1}\Big\}. \]
 Given a spanning tree $\F$ and two vertices $x$ and $y$ in $G$, we write $\Gamma_\F(x,y)$ for the unique path in $\F$ connecting $x$ and $y$. In the following theorem, the graph $G$ satisfies the conditions of \cref{T:mainthm} and so the path $\Gamma_\F$ is  well defined almost surely.

\begin{figure}[t]
\centering
\includegraphics[trim = 4.915cm 4.05cm 4.915cm 4.05cm, clip, height=0.4\textwidth]{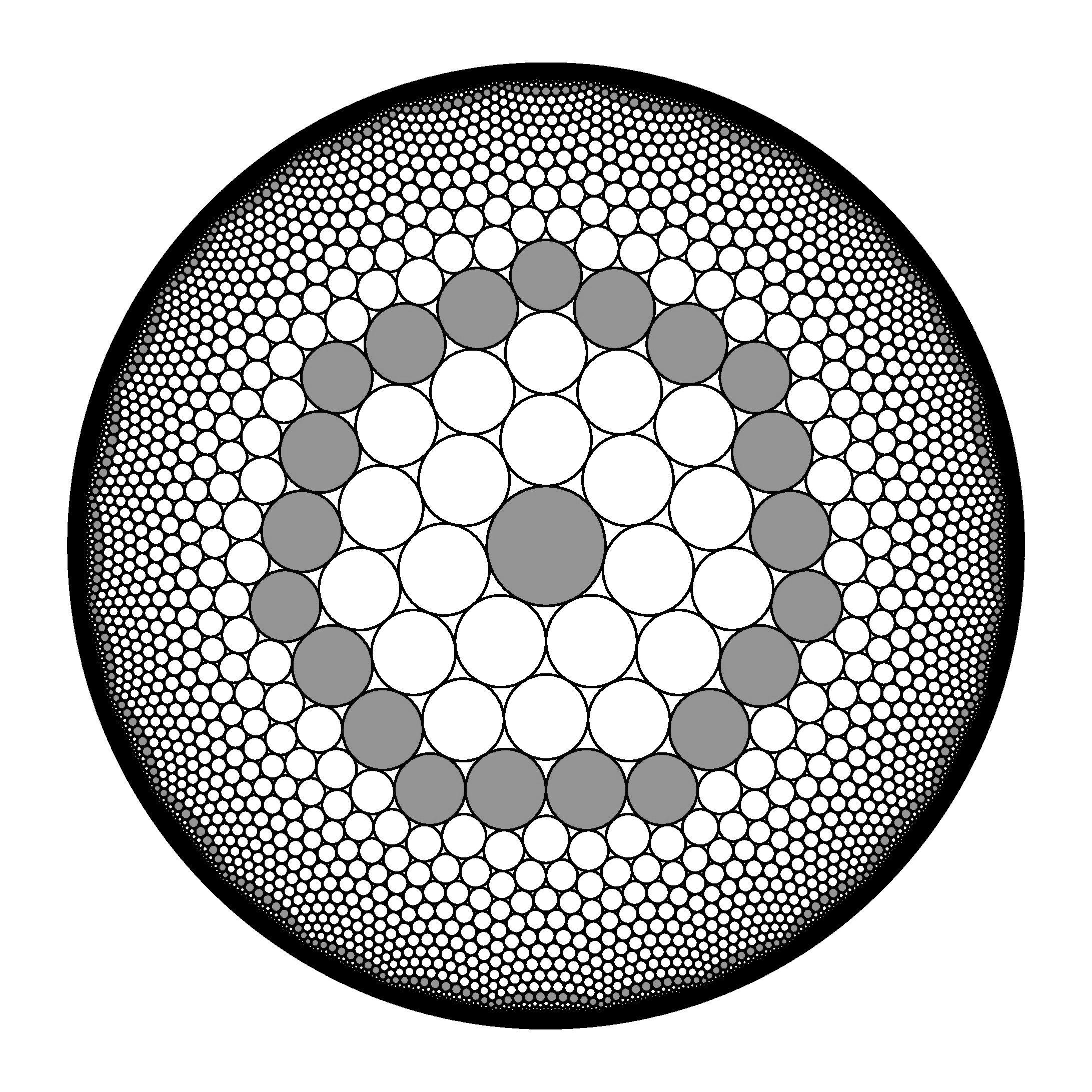} \hspace{1cm}
\includegraphics[trim =0.5cm 0.225cm 0.4cm 0.18cm, clip, height=0.4\textwidth, angle=-1.5]{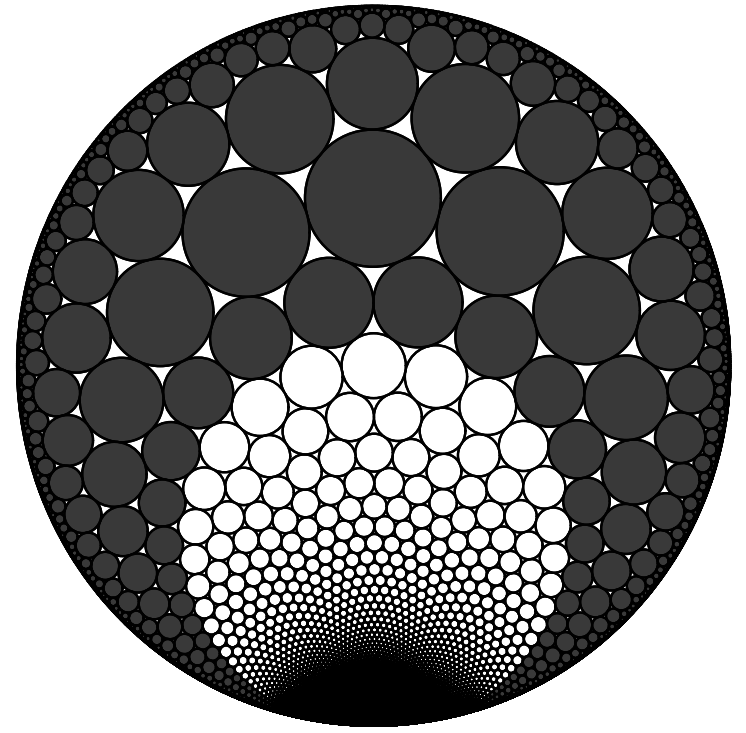}
\caption{Two bounded degree, simple, proper plane triangulations for which the graph distance is not comparable to the hyperbolic distance. Similar examples are given in \cite[Figure 17.7]{St05}. Left: In this example, rings of degree seven vertices (grey) are separated by growing bands of degree six vertices (white), causing the hyperbolic radii of circles to decay. The bands of degree six vertices can grow surprisingly quickly without the triangulation becoming recurrent \cite{siders1998layered}. Right: In this example, half-spaces of the $8$-regular (grey) and $6$-regular (white) triangulations have been glued together along their boundaries;  the circles corresponding to the $6$-regular half-space are contained inside a horodisc and have decaying hyperbolic radii.}
\label{fig:layers}
\end{figure}

\begin{thm}[Free diameter exponent]\label{T:Ftailcp}
Let $G$ be a transient, polyhedral, proper plane network with \asaf{$\bM_G < \infty$},
let $\F$ be the free uniform spanning forest of $G$, and let $e=(x,y)$ be an edge of $G$. Then there exist positive  constants $k_1=k_1(\bM,r_\bbH(x))$, increasing in $r_\bbH(x)$, and $k_2=k_2(\bM)$ such that
\vspace{0.25em}
\[ {k_1}R^{-1} \leq \P\left(\diam_\bbH(\Gamma_\F(x,y)) \geq R\right)  \leq  k_2 R^{-1}  \vspace{0.25em}\]
for every $R\geq 1$.
\end{thm}

Given a spanning forest $\F$ of $G$ in which every component is a one-ended tree and an edge $e$ of $G$, the \textbf{past} of $e$ in $\F$, denoted $\past_\F(e)$, is defined to be the unique finite connected component of $\F \setminus \{e\}$ if $e \in \F$ and to be the empty set otherwise. The following theorem is equivalent to \cref{T:Ftailcp} by duality (see \cref{Sec:Duality,Subsec:dualexponents}).

\begin{thm}[Wired diameter exponent]\label{T:Wtailcp}
Let $G$ be a transient, polyhedral, proper plane network with  \asaf{$\bM_G < \infty$},
let $\F$ be the wired uniform spanning forest of $G$, and let $e=(x,y)$ be an edge of $G$.
Then there exist positive constants $k_1=k_1(\bM,r_\bbH(x))$, increasing in $r_\bbH(x)$, and $k_2=k_2(\bM)$ such that
\vspace{.25em}
\[k_1 R^{-1} \leq \P \left(\diam_\bbH (\past_\F(e)) \geq R\right) \leq k_2 R^{-1}\vspace{.25em} \]
 for every $R\geq 1$.

\end{thm}

By similar methods, we are also able to obtain a universal exponent of $1/2$ for the tail of the area of the past of an edge in the WUSF.

\begin{thm}[Wired area exponent]\label{Thm:Wareacp}
Let $G$ be a transient, polyhedral, proper plane network with  \asaf{$\bM_G < \infty$},
let $\F$ be the wired uniform spanning forest of $G$, and let $e=(x,y)$ be an edge of $G$.
Then there exist positive constants $k_1=k_1(\bM,r_\bbH(x))$, increasing in $r_\bbH(x)$, and $k_2=k_2(\bM)$ such that
\vspace{.25em}
\[k_1 R^{-1/2} \leq \P \left(\area_\bbH (\past_\F(e)) \geq R\right) \leq k_2 R^{-1/2} \vspace{.25em}\]
 for every $R\geq 1$.
\end{thm}

The exponents $1$ and $1/2$ occurring in \cref{T:Wtailcp,Thm:Wareacp} should, respectively, be compared with the analogous exponents for the survival time and total progeny of a critical branching process whose offspring distribution has finite variance (see e.g.\ \cite[{\S\!}~5,12]{LP:book}).

\medskip

For \emph{uniformly transient} proper plane graphs (that is, proper plane graphs in which the escape probabilities of the random walk are bounded uniformly away from zero), the hyperbolic radii of circles in $(P,P^\dagger)$ are bounded away from zero uniformly (\cref{Prop:radii}). This implies that
 the hyperbolic metric and the graph metric are rough-isometric (\cref{cor:rough}).
Consequently, \cref{T:Ftailcp,T:Wtailcp,Thm:Wareacp} hold with the graph distance and counting measure as appropriate (\cref{Cor:graphexponents}).
 This yields the following appealing corollary, which applies, for example, to planar Cayley graphs of co-compact Fuchsian groups.

\begin{corollary}[Free length exponent]\label{cor:Flengthtail}
Let $G$ be a uniformly transient, polyhedral, proper plane network with  \asaf{$\bM_G < \infty$} 
and let $\F$ be the free uniform spanning forest of $G$.
Let $\mathbf{p}>0$ be a uniform lower bound on the escape probabilities of $G$. Then there exist positive  constants $k_1=k_1(\bM,\mathbf{p})$ and $k_2=k_2(\bM,\mathbf{p})$ such that
\vspace{.25em}
\[ k_1 R^{-1/2} \leq \P\left(|\Gamma_\F(x,y)| \geq R\right)  \leq k_2 R^{-1/2}
\vspace{.25em}
\]
for every edge $e=(x,y)$ of $G$ and every $R\geq 1$.
\end{corollary}

See \cite{kenyon2000asymptotic,barlow2011spectral,LMS08,masson2008growth,shiraishi2013growth,bhupatiraju2016inequalities} and the survey \cite{barlow2014loop} for related results on Euclidean lattices.

\subsection*{Organisation}

\cref{Sec:background} contains definitions and background on those notions that will be used throughout the paper. Experienced readers are advised that this section also includes proofs of a few simple folklore-type lemmas and propositions. \cref{Sec:mainproof} contains the proofs of \cref{T:mainthm,T:planeWUSF}, as well as the upper bound of \cref{T:Wtailcp} in the case that $G$ is a triangulation. \cref{Sec:lowerbound} completes the proofs of \cref{T:Ftailcp,T:Wtailcp,Thm:Wareacp,cor:Flengthtail}; the most substantial component of this section is the proof of the lower bound of \cref{T:Wtailcp}. \cref{Sec:lowerbound} also includes the statement and proof of an extension of the Ring Lemma to double circle packings. 
We conclude with a remark and an open problem in \cref{sec:conclusion}.

\section{Background and definitions}\label{Sec:background}

\subsection{Notation}
We write $E^\rightarrow$ for the set of oriented edges of a network $G=(V,E)$. An oriented edge $e\in E^\rightarrow$ is oriented from its tail $e^-$ to its head $e^+$, and has reversal $-e$.

For each $r',r>0$, and $z \in \C$, we define the open ball and closed annulus
\[ B_{z}(r) = \{ z' \in \C : |z-z'| < r \}
\quad \text{and} \quad
 A_{z}(r,r')=\{ z' \in \C : r \leq |z'-z| \leq r' \}. \]
We write $d_\C$ for the Euclidean metric on $\C$ and write $d_\bbH$ for the hyperbolic metric on $\D$.
\subsection{Uniform Spanning Forests}\label{subsec:usfbackground}
We begin with a succinct review of some basic facts about USFs, referring the reader to \cite{BLPS} and Chapters 4 and 10 of \cite{LP:book} for a comprehensive overview.
 For each finite, connected graph $G$, we define $\UST_G$ to be the uniform measure on the set of spanning trees of $G$ (i.e., connected subgraphs of $G$ containing every vertex and no cycles). More generally, for a finite network $G=(G,w)$, we define $\UST_{G}$ to be the probability measure on spanning trees of $G$ for which the measure of each tree is proportional to the product of the conductances of the edges in the tree.

An \textbf{exhaustion} of an infinite network $G$ is a sequence $\langle V_n\rangle_{n\geq0}$ of finite, connected subsets of $V$ such that $V_n \subseteq V_{n+1}$ for all $n\geq 0$ and $\cup_n V_n = V$. Given such an exhaustion, let the network $G_n$ be the subgraph of $G$ induced by $V_n$ together with the conductances inherited from $G$. The \textbf{free uniform spanning forest} measure $\FUSF_G$ is defined to be the weak limit of the sequence  $\langle \UST_{G_n}\rangle_{n\geq1}$, so that
\[ \FUSF_G(S\subset \F) = \lim_{n\to\infty} \UST_{G_n}(S\subset T). \]
for each finite set $S\subset E$, where $\F$ is a sample of $\FUSF_G$ and $T$ is a sample of $\UST_{G_n}$. For each $n$, we also construct a network $G_n^*$ from $G$ by gluing (wiring) every vertex of $G \setminus G_n$ into a single vertex, denoted $\partial_n$, and deleting all the self-loops that are created. We identify the set of edges of $G_n^*$ with the set of edges of $G$ that have at least one endpoint in $V_n$. The \textbf{wired uniform spanning forest} measure $\WUSF_G$ is defined to be the weak limit of the sequence  $\langle \UST_{G_n^*}\rangle_{n\geq1}$, so that
\[ \WUSF_G(S\subset \F) = \lim_{n\to\infty} \UST_{G_n^*}(S\subset T) \]
for each finite set $S\subset E$, where $\F$ is a sample of $\WUSF_G$ and $T$ is a sample of $\UST_{G_n^*}$.
 These weak limits were both implicitly proven to exist by Pemantle~\cite{Pem91}, although the WUSF was not considered explicitly until the work of H\"aggstr\"om \cite{Hagg95}.  Both measures are easily seen to be concentrated on the set of \textbf{essential spanning forests} of $G$, that is, cycle-free subgraphs of $G$ including every vertex such that every connected component is infinite. The measure $\FUSF_G$ stochastically dominates $\WUSF_G$ for every infinite network $G$.

\subsubsection{The Spatial Markov Property} \label{subsec:Markov}
Let $G=(V,E)$ be a finite or infinite  network and let $A$ and $B$ be subsets of $E$. We write $(G-B)/A$ for the (possibly disconnected) network formed from $G$ by deleting every edge in $B$ and contracting (i.e., identifying the two endpoints of) every edge in $A$. We identify the edges of $(G-B)/A$ with $E\setminus B$.
Suppose that $G$ is finite, and that
\[ \UST_G(A \subseteq T, B \cap T =\emptyset)>0.\]
Then, given the event that $T$ contains every edge in $A$ and none of the edges in $B$, the conditional distribution of $T$ is equal to the union of $A$ and the UST of $(G-B)/A$. That is, for every event $\sA\subseteq \{0,1\}^E$,
\[ \UST_G(
T \in \sA \mid A \subset
T,\, B\cap
T =\emptyset) = \UST_{(G-B)/A}(
T \cup A \in \sA).\]

\noindent This is the \textbf{spatial Markov property} of the UST. Taking limits over exhaustions, we obtain a corresponding spatial Markov property for the USFs: If $G=(V,E)$ is an infinite network and $A$ and $B$ are subsets of $E$ such that
\[ \WUSF_G(A \subseteq \F, B \cap \F =\emptyset)>0 \]
and the network $(G-B)/A$ is locally finite, then
\begin{align} \WUSF_G(\F \in \sA \mid A \subset \F,\, B\cap \F =\emptyset) &= \WUSF_{(G-B)/A}(\F \cup A \in \sA). \end{align}
(If $(G-B)/A$ is not connected, we define $\WUSF_{(G-B)/A}$ to be the product of the WUSF measures on the connected components of $(G-B)/A$.)
A similar spatial Markov property holds for the FUSF.

The UST and USFs also enjoy a \emph{strong} form of the spatial Markov property. Let $G$ be a finite or infinite network, and suppose that $H$ is a random element of $\{0,1\}^E$, which we think of as a random subgraph of $G$. We say that a random element $K$ of $\{0,1\}^E$, defined on the same probability space as $H$, is a \textbf{local set for $H$} if for every set $W \subseteq E$, the event $\{K \subseteq W\}$ is, up to a null set, measurable with respect to the $\sigma$-algebra generated by the collection of random variables $\{ H(e) : e \in W \}$.

 For example, suppose that $\F$ is a random spanning forest of a network  $G$, and that $u$ and $v$ are vertices of $G$. Let $K=K(\F)$ be the random set of edges that is equal to the path connecting $u$ and $v$ in $\F$ if such a path exists, and otherwise equal to all of $E$.  Then $K$ is a local set for $\F$, since, for every proper subset $W$ of $E$, the event $\{K \subseteq W\}$ is equal (up to the null event in which $\F$ is not a forest) to the event that $u$ and $v$ are in the same component of $\F$ and that the unique path connecting $u$ and $v$ in $\F$ is contained in $W$, which is clearly measurable with respect to the restriction of $\F$ to $W$.

An example of a random set $K$ that is not a local set for $\F$ is, for instance, the edges of $\F$ incident to a given vertex $v$. If $v$ has degree at least $2$ and $e$ is some edge incident to $v$, then the event $\{K \subset \{e\}\}$ cannot be determined just by knowing $\F(e)$. Note also that \cref{prop:Markov} does not hold for this $K$, since if we know which edges are included in $K$ then we also know that $\F(e)=0$ for all edges $e \not \in K$ incident to $v$.

Given a random subgraph $H$ of a network $G$ and a local set $K$ for $H$, we write
$\cF_K$ to denote the $\sigma$-algebra generated by $K$ and the random collection of random variables $\{ H(e) : e \in K \}$. We also write $K_o = \{e \in E : e \in K,\, H(e)=1\}$ and $K_c = K \setminus K_o$ for the sets of edges in $K$ that are included (open) in $H$ and not included (closed) in $H$ respectively.

\begin{prop}[Strong Spatial Markov Property for the WUSF]\label{prop:Markov} Let $G=(V,E)$ be an infinite network, let $\F$ be a sample of the wired uniform spanning forest of $G$, and let $K$ be a local set for $\F$ that is either finite or equal to all of $E$ almost surely. Conditional on $\cF_K$ and the event that $K$ is finite, let $\hat \F$ be a sample of the wired uniform spanning forest of $(G-K_c)/K_o$. Let $\F' = K_o \cup \hat \F$ if $K_o$ is finite and $\F'=\F$ if $K_o=E$.
Then $\F$ and $\F'$ have the same distribution.
\end{prop}

\begin{proof}
 Expand the conditional probability
\begin{multline*}\WUSF_G\left(\F \in \sA \mid \cF_K\right)\mathbbm{1}(K\neq E)\\ = \sum_{W_1,W_2 \subset E \text{ finite}}\hspace{-1.7em}\WUSF_G\left(\F \in \sA \mid K_o = W_1, K_c=W_2\right)\mathbbm{1}\left(K_o = W_1, K_c = W_2\right).\end{multline*}
Since $K$ is a local set for $\F$, the right-hand side is equal to
\begin{align*}
 \sum_{W_1,W_2 \subset E \text{ finite}}\hspace{-1.7em}\WUSF_G\left(\F \in \sA \mid W_1 \subseteq \F, W_2\cap \F =\emptyset\right) \mathbbm{1}\left(K_o = W_1, K_c = W_2\right). \end{align*}
Applying the spatial Markov property gives
\begin{multline*}\WUSF_G\left(\F \in \sA \mid \cF_K\right)\mathbbm{1}(K\neq E)\\ = \sum_{W_1,W_2 \subset E \text{ finite}}\hspace{-1.7em}\WUSF_{\left(G-W_2\right)/W_1}\left(\hat \F\cup W_1 \in \sA\right)\mathbbm{1}\left(K_o = W_1, K_c = W_2\right) \end{multline*}
\[
\hspace{-2.525cm}= \WUSF_{\left(G-K_c\right)/K_o}\left(\hat \F \cup K_o \in \sA\right)\mathbbm{1}(K\neq E), \]
from which the claim follows immediately.
\end{proof}

Similar strong spatial Markov properties hold for the FUSF and UST, and admit very similar proofs.

\subsection{Random walk, effective resistances}
Given a network $G$ and a vertex $u$ of $G$, we write $\bP^G_u$ for the law of the simple random walk on $G$ started at $u$, and will often write simply $\bP_u$ if the choice of network is unambiguous.
% and write $\bP^{G_d^N}_v$ for the law of the simple random walk on $G_d^N$ started at $v$.
For each set of vertices $A$, we let $\tau_A$  be the first time  the random walk visits $A$, letting $\tau_A=\infty$ if the walk never visits $A$. Similarly, $\tau_A^+$ is defined to be the first positive time that the random walk visits $A$. The conductance $c(u)$ of a vertex $u$ is defined to be the sum of the conductances of the edges emanating from $u$.

Let $A$ and $B$ be sets of vertices in a finite network $G$. The \textbf{effective conductance} between $A$ and $B$ in $G$ is defined to be
\[\Ceff(A\leftrightarrow B; G) = \sum_{v\in A}c(v)\bP_v(\tau_B < \tau^+_A),\]
while the effective resistance $\Reff(A\leftrightarrow B; G)$ is defined to be the reciprocal of the effective conductance,
$\Reff(A\leftrightarrow B; G)= \Ceff(A\leftrightarrow B; G)^{-1}$.
Now suppose that $G$ is an infinite network with exhaustion $\langle V_n\rangle_{n\geq0}$ and let $A$ and $B$ be finite subsets of $V$. Let $\langle G_n \rangle_{n\geq0}$ and $\langle G_n^* \rangle_{n\geq0}$ be defined as in \cref{subsec:usfbackground}.
The \textbf{free effective resistance between $A$ and $B$} is defined to be the limit
\[\ReffF(A \leftrightarrow B;\, G) = \lim_{n\to\infty} \Reff(A \leftrightarrow B;\, G_n),\]
% which is independent of the choice of exhaustion.
% The wired effective resistance between $A$ and $B$ may also be defined as a limit over exhaustions as follows. For each $n$, we construct a network $G_n^*$ from $G$ by gluing every vertex of $G \setminus G_n$ into a single vertex $\partial_n$, and deleting all the loops that are created. The wired effective resistance between $A$ and $B$ is equal to the limit
while the \textbf{wired effective resistance between $A$ and $B$} is defined to be
\[\ReffW(A \leftrightarrow B;\, G) = \lim_{n\to\infty} \Reff(A \leftrightarrow B;\, G_n^*).\]
Free and wired effective conductances are defined by taking reciprocals. The free effective resistance between two, possibly infinite, sets $A$ and $B$ is defined to be the limit of the free effective resistances between $A \cap V_n$ and $B \cap V_n$, which are decreasing in $n$.
We also define
% \[\Reff(A\to\infty) = \lim_{n\to\infty} \Reff(A \to \partial_n ; G^*_n)\]
% and
\[\Reff\left(A\leftrightarrow B\cup\{\infty\}\right) = \lim_{n\to\infty} \Reff\left(A \leftrightarrow B\cup\{\partial_n\} ; G^*_n\right).\]
See e.g.\ \cite{LP:book} for further background on electrical networks.
% Taking limits in Kirchhoff's effective resistance formula, we have

We will make frequent use of \textbf{Rayleigh' monotonicity principle} \cite[Chapter 2.4]{LP:book}, which states that the effective conductance between any two sets in a network is an increasing function of the edge conductances. In particular, it follows that the effective conductance between two sets decreases when edges are deleted from the network (which corresponds to taking the conductance of those edges to zero), and increases when edges are contracted (which corresponds to taking the conductance of those edges to infinity).

The proof of \cref{T:planeWUSF} will require the following simple lemma.

\begin{lemma}\label{Lem:triangle} Let $A$ and $B$ be sets of vertices in an infinite network $G$. Then
\[ \ReffW(A \leftrightarrow B;\, G) \leq 3\max\left\{\Reff\left(A \leftrightarrow B \cup \{\infty\};\, G\right),\, \Reff\left(B \leftrightarrow A \cup \{\infty\};\, G\right)\right\}.\]
\end{lemma}
\begin{proof}
% It suffices to consider the case that $A$ and $B$ are both finite. Let $\mu_A$ be a probability measure on simple paths starting in $A$ that are either infinite or finite and end in $B$ and such that
% \[\cE(\mu_A)=\Reff(A \leftrightarrow B\cup\{\infty\};G).\]\frac{M\ReffW(A\leftrightarrow B;G)}{\ReffW(A\leftrightarrow B;G)-M}
% Similarly, let  $\mu_B$ be a probability measure on simple paths starting in $B$ that are either infinite or finite and end in $A$ and such that
% \[\cE(\mu_B)=\Reff(B \leftrightarrow A\cup\{\infty\};G).\]
% Decompose $\mu_A$ as the convex combination $\mu_A=\lambda_A\mu_A^1+(1-\lambda_A)\mu_A^2$ where $\mu_A^1$ is defined by . Then
% \[\sE(\mu_A)^{1/2} \leq \lambda_A\sE(\mu_A^1)^{1/2} + (1-\lambda_A)\sE(\mu_A^2)^{1/2} \leq 2\sE(\mu_A)^{1/2} \]
Let $M=\max\left\{\Reff(A \leftrightarrow B \cup \{\infty\};\, G),\, \Reff(B \leftrightarrow A \cup \{\infty\};\, G)\right\}$.
Recall that for any three sets $A$, $B$ and $C$ in $G_n^*$ (or any other finite network) \cite[Exercise 2.33]{LP:book},
\[ \Reff(A \leftrightarrow B \cup C; G_n^*)^{-1} \leq \Reff(A \leftrightarrow B; G_n^*)^{-1} + \Reff(A \leftrightarrow  C; G_n^*)^{-1}. \]
Letting $C=\{\partial_n\}$ and taking the limit as $n\to\infty$, we obtain that
\[ M^{-1} \leq \Reff(A \leftrightarrow B \cup \{\infty\}; G)^{-1} \leq \ReffW(A \leftrightarrow B; G)^{-1} + \Reff(A \leftrightarrow  \infty; G)^{-1}. \]
Rearranging, we have that
\begin{equation*}\label{eq:triangle1}\Reff(A\leftrightarrow\infty;G) \leq \frac{M\ReffW(A\leftrightarrow B;G)}{\ReffW(A\leftrightarrow B;G)-M}.\end{equation*}
% and similarly
By symmetry, the inequality continues to hold when we exchange the roles of $A$ and $B$. Combining both of these inequalities
% \begin{equation}\label{eq:triangle2}\Reff(B\to\infty;G) \leq \frac{M\ReffW(A\leftrightarrow B;G)}{\ReffW(A\leftrightarrow B;G)-M}.\end{equation}
% Combining \eqref{eq:triangle1} and \eqref{eq:triangle2} with
 with the triangle inequality for effective resistances \cite[Exercise 9.29]{LP:book} yields that
\begin{align*}
\ReffW(A\leftrightarrow B;G) &\leq \Reff(A\leftrightarrow\infty;G)+\Reff(B\leftrightarrow\infty;G)\leq 2\frac{M\ReffW(A\leftrightarrow B;G)}{\ReffW(A\leftrightarrow B;G)-M},
\end{align*}
which rearranges to give the claimed inequality.
\end{proof}

% The wired effective resistance between $A$ and $B$ can be defined equivalently by the variational formula
% \[ \ReffW(A \leftrightarrow B ;\, G) = \min \{ \cE(\theta) : \theta \text{ a unit flow from $A$ to $B$ in $G$}\} \]
% just as for the effective resistance in a finite network \cite[Exercise 9.26]{LP:book}. While the free effective resistance is also given by a variational formula, for our pruposes the following inequality suffices. The \textbf{support} of a flow $\theta$ is defined to be the set of oriented edges $e$ such that $\theta(e)\neq0$. A flow is said to be \textbf{finitely supported} if its support is finite.
% \[ \ReffF(A \leftrightarrow B;\, G) \leq \inf\{ \cE(\theta) : \theta \text{ a finitely supported unit flow from $A$ to $B$ in $G$}\}.\]

% \medskip

% We also define the effective resistance from set $A$ to $B\cup\{\infty\}$. An antisymmetric edge function $\theta$ is a \textbf{unit flow from $A$ to $B \cup \{\infty\}$} if
% \[\sum_{v \in A}\nabla\theta(v)=1 \quad \text{ and } \quad  \nabla\theta(v) =0 \quad \forall v \in V\setminus(A\cup B).\]
% We define the \textbf{effective resistance from $A$ to $B\cup\{\infty\}$} to be
% \[ \Reff(A \to B \cup \{\infty\};\, G) = \inf \{ \cE(\theta) : \theta \text{ a unit flow from $A$ to $B\cup\{\infty\}$ in $G$}\}. \]

% \note{define conductances?}

% \medskip
\subsubsection{Kirchhoff's Effective Resistance Formula.}
The connection between effective resistances and spanning trees was first discovered by Kirchhoff \cite{kirchhoff1847ueber} (see also \cite{BuPe93}). 
\begin{thm}[Kirchhoff's Effective Resistance Formula]\label{Thm:Kirchhoff} Let $G$ be a finite network. Then for every $e\in E^\rightarrow$
\[ \UST_G(e \in T) = c(e)\Reff(e^- \leftrightarrow e^+;G).\]
\end{thm}
\noindent Taking limits over exhaustions, we also have the following extension of Kirchhoff's formula:
\begin{align}
\WUSF_G(e\in \fF) = c(e)\ReffW(e^-\leftrightarrow e^+;G).\end{align}
A similar equality holds for the FUSF.

\subsubsection{The method of random paths.}

We say that a path $\Gamma$ in a network $G$ is \textbf{simple} if it does not visit any vertex more than once.
Given  a probability measure $\nu$  on simple paths $\Gamma$ in a network $G$,
  we define the \textbf{energy} of $\nu$ to be
\[\cE(\nu)=\frac{1}{2}\sum_{e\in E^\rightarrow}\frac{1}{c(e)}\big(\nu(e \in \Gamma)-\nu(-e\in\Gamma)\big)^2.\]
Effective resistances can be bounded from above by energies of random paths: In particular, if $G$ is an infinite network and $A$ and $B$ are two finite sets of vertices in $G$, then
\[\Reff(A \leftrightarrow B\cup\{\infty\} ; G) = \min\left\{\cE(\nu) :\begin{array}{l} \nu \text{ a probability measure on  simple}\\  \text{paths in $G$ starting in $A$ that are}\\ \text{either infinite or finite and end in $B$}\end{array}\right\}\]
See e.g.\ \cite[\S3]{LP:book} and \cite{peres1999probability} for more detail.
Obtaining resistance bounds by defining flows using random paths in this manner is referred to as the \textbf{method of random paths}.

It will be convenient to use the following weakening of the method of random paths.
 Given the law $\mu$ of a random subset $W  \subset V(G)$, define
 \[\cE(\mu) = \sum_{v\in V}\mu(v\in W)^2.\]

\begin{lem}[Method of random sets]\label{lem:randomsets} Let $A$ and $B$ be two finite sets of vertices in an infinite network $G$, and let $\mu$ be a measure on subsets $W  \subset V(G)$ such that the subgraph of $G$ induced by $V$ almost surely contains a path starting at $A$ that is either infinite or finite and ends at $B$. Then
\begin{equation}\label{eq:randomsets}\Reff(A \leftrightarrow B \cup \{\infty\}; G) \leq \sup_{e\in E}c(e)^{-1} \cE(\mu). \end{equation}
\end{lem}

\begin{proof}
Given $W$, let $\Gamma$ be a simple path connecting $A$ to $B$ that is contained in $W$. Then, letting $\nu$ be the law of $\Gamma$,
% \[\cE(\nu) \]
\begin{align*}
	\cE(\nu)
	& \leq \sup_{e}\frac{1}{c(e)}\sum_{e\in E^\rightarrow} \nu(e \in \Gamma)^2.
	\end{align*}
Letting $\Gamma'$ be an independent random path with the same law as $\Gamma$, the sum above is exactly the expected number of oriented edges that are used by both $\Gamma$ and $\Gamma'$. Since these paths are simple, they each contain at most one oriented edge emanating from $v$ for each vertex $v\in V$. It follows that the number of oriented edges included in both paths is at most the number of vertices included in both paths. This yields that
\[ \mathcal{E}(\nu) \leq \sup_{e\in E}\frac{1}{c(e)} \sum_{v\in V} \nu(v \in \Gamma)^2\leq \sup_{e\in E}\frac{1}{c(e)} \sum_{v\in V} \mu(v \in W)^2 =  \sup_{e\in E}\frac{1}{c(e)} \cE(\mu).\qedhere\]
\end{proof}

\subsection{Plane Graphs and their USFs}\label{Sec:Embeddings}

Given a graph $G=(V,E)$, let $\mathbb{G}$ be the metric space defined as follows. For each edge $e$ of $G$, choose an orientation of $e$ arbitrarily and let $\{I(e) : e \in E\}$ be a set of disjoint isometric copies of the interval $[0,1]$.
The metric space $\mathbb{G}$ is defined as a quotient of the union $\bigcup_e I(e) \cup V$, where we identify the endpoints of $I(e)$ with the vertices $e^-$ and $e^+$ respectively, and is equipped with the path metric.

Let $S$ be an orientable surface without boundary, which in this paper will always be a domain $D \subseteq \C\cup\{\infty\}$.
A \textbf{proper embedding} of a graph $G$ into $S$ is a continuous, injective map $z:\mathbb{G}\to S$ satisfying the following conditions:
\begin{enumerate}
\item (\textbf{Every face is a topological disc.}) Every connected component of the complement $S\setminus z(\bbG)$, called a \textbf{face} of $(G,z)$, is homeomorphic to the disc. Moreover, for each connected component $U$ of $S \setminus z(\bbG)$, the set of oriented edges of $G$ that have their left-hand side incident to $U$ forms either a cycle or a bi-infinite path in $G$.
% \item (\textbf{Boundaries are paths.}) For each connected component $U$ of $S \setminus z(\bbG)$, the set of oriented edges of $G$ that have their left-hand side incident to $U$ forms either a cycle or a bi-infinite path in $G$. (This condition is automatically satisfied if $S$ is simply connected or if every face is incident to only finitely many edges.)
 % The embedding is \textbf{locally finite} (or \textbf{proper}) if
 \item (\textbf{$z$ is locally finite.}) Every compact subset of $S$ intersects at most finitely many edges of $z(\mathbb{G})$. Equivalently, the preimage $z^{-1}(K)$ of every compact set $K\subseteq S$ is compact in $\mathbb{G}$.
 % \item
 \end{enumerate}
A locally finite, connected graph is planar if and only if it admits a proper embedding into some domain $D\subseteq \C\cup\{\infty\}$.
 % A \textbf{simply connected planar graph} is a graph that admits a locally finite embedding into $\C\cup\{\infty\}$ or $\C$.
 A \textbf{plane graph} $G=(G,z)$ is a planar graph $G$ together with a specified embedding $z:\mathbb{G}\to D \subseteq \C\cup\{\infty\}$.
A \textbf{plane network} $G=(G,z,c)$ is a planar graph together with a specified embedding and an assignment of positive conductances $c: E \to (0,\infty)$.

 Given a pair $G=(G,z)$ of a graph together with a proper embedding $z$ of $G$ into a domain $D$, the \textbf{dual} $G^\dagger$ of $G$ is the graph that has the faces of $G$ as vertices, and has an edge drawn between two faces of $G$ for each edge incident to both of the faces in $G$. By drawing each vertex of $G^\dagger$ in the interior of the corresponding face of $G$ and each edge of $G^\dagger$ so that it crosses the corresponding edge of $G$ but no others, we obtain an embedding $z^\dagger$ of $G^\dagger$ in $D$. The edge sets of $G$ and $G^\dagger$ are in natural correspondence, and we write $e^\dagger$ for the edge of $G^\dagger$ corresponding to $e$. If $e$ is oriented, we let $e^\dagger$ be oriented so that it crosses $e$ from right to left as viewed from the orientation of $e$.
% A \textbf{face} of a plane graph $(G,z)$ is a connected component of the  complement of $\z(\bbG)$.
If $G=(G,z,c)$ is a plane network, we assign the conductances $c^\dagger(e^\dagger)=c(e)^{-1}$ to the edges of $G^\dagger$.

\subsubsection{USF Duality.}\label{Sec:Duality}
Let $G$ be a plane network with dual $G^\dagger$. For each set $W \subseteq E$, let $W^\dagger:=\{e^\dagger : e \notin W\}$. If $G$ is finite and $t$ is a spanning tree of $G$, then $t^\dagger$ is a spanning tree of $G^\dagger$: the subgraph $t^\dagger$ is connected because $t$ has no cycles, and has no cycles because $t$ is connected. Moreover, the ratio
% \[\frac{\UST_G(t)}{\UST_{G^\dagger}(t^\dagger)}=
\[\frac{\prod_{e\in t}c(e)}{\prod_{e^\dagger\in t^\dagger}c^\dagger(e^\dagger)}= \prod_{e\in E}c(e)\]
does not depend on $t$. It follows that if $T$ is a random spanning tree of $G$ with law $\UST_G$, then $T^\dagger$ is a random spanning tree of $G^\dagger$ with law $\UST_{G^\dagger}$.
 This duality was extended to infinite proper plane networks by BLPS.
\begin{thm}[{\cite[{Theorem 12.2 and Proposition 12.5}]{BLPS}}]
\label{thm:duality}
 Suppose that $G$ is an infinite proper plane network with locally finite dual $G^\dagger$. Then if $\F$ is a sample of $\FUSF_G$, the subgraph $\F^\dagger$ is an essential spanning forest of $G^\dagger$ with law $\WUSF_{G^\dagger}$. In particular, $\F$ is connected almost surely if and only if every component of $\F^\dagger$ is one-ended almost surely.
 \end{thm}

 %In \cref{Sec:TUSF} we extend this duality to plane graphs that are not proper by introducing the \emph{transboundary uniform spanning forest}. For an alternative approach to duality in this generalised setting, see \cite{Lyons09}.

% On planar graphs, the USFs enjoy the following duality relation \cite{BLPS}: If $G$ is a plane graph with locally finite dual $G^\dagger$ and $\F$ is a sample of $\FUSF_M$, then the set
% \[ \F^\dagger = \{ e^\dagger \in E(G^\dagger) : e \notin \F \} \]
% is a spanning forest of $G^\dagger$ with the law of $\WUSF_{G^\dagger}$.

\subsection{Circle packing}\label{Sec:CP}

% Our proof makes essential use of the the theory of circle packing.
% A \textbf{circle packing} $P$ is a set of (geometric) discs in the Riemann sphere $\C\cup\{\infty\}$ that have disjoint interiors (i.e.~do not overlap) but can be tangent. The \textbf{tangency graph} $G=G(P)$ of a circle packing $P$ is the plane graph with the centres of circles in $P$ as its vertices and with edges given by straight lines between the centres of tangent circles. The Koebe-Andreev-Thurston Circle Packing Theorem \cite{K36,Th78} states that every finite simple planar graph arises as the tangency graph of a circle packing, and that if the graph is a triangulation then its circle packing is unique up to M\"obius transformations and reflections.

% The \textbf{carrier} of a circle packing $P$ is defined to be the union over

% While straightforward compactness arguments imply the existence of a circle packing of any simple planar triangulation, the packing is no longer unique up to Mobius transformations.
% Recall that a \textbf{circle packing} $P$ is a set of (geometric) discs in the Riemann sphere $\C\cup\{\infty\}$ that have disjoint interiors but can be tangent. The \textbf{tangency graph} $G=G(P)$ of a circle packing $P$ is the plane graph with the centres of circles in $P$ as its vertices and with edges given by straight lines between the centres of tangent circles.
% Most of the background on circle packing we will require was alrea

We now give some background on circle packing.
% Recall the backgroun
The \textbf{carrier} of a circle packing $P$, denoted carr$(P)$, is the union of all the discs of $P$ and of all the faces of $G(P)$, so that the embedding $z$ of $G(P)$ defined by drawing straight lines between the centres of tangent circles is a proper embedding of $G(P)$ into carr$(P)$. Similarly, the carrier of a double circle packing $(P,P^\dagger)$ is defined to be the union of all the discs in $P\cup P^\dagger$. Given a domain $D\subset \C\cup\{\infty\}$, a circle packing $P$ (or double circle packing $(P,P^\dagger)$) is said to be \textbf{in $D$} if its carrier is $D$. In particular, a (double) circle packing $P$ is said to be \textbf{in the plane} if its carrier is the plane $\C$ and \textbf{in the disc} if its carrier is the open unit disc $\D$. The following theorems, which we stated in the introduction, are the cornerstones of the theory for infinite proper plane graphs.

\begin{thm}[He-Schramm Existence and Uniqueness Theorem \cite{Schramm91,HS93,HeSc,he1999rigidity}] \label{dblCP1}
Let $G$ be an infinite, polyhedral, proper plane graph with locally finite dual.
Then $G$ admits a double circle packing either in the plane or the disc, but not both, and this packing is unique up to M\"obius transformations and reflections of the plane or the disc as appropriate.
\end{thm}

\begin{thm}[He-Schramm Recurrence Theorem \cite{HeSc,he1999rigidity}] \label{dblCP2}
Let $G$ be an infinite, polyhedral, proper plane graph with bounded degrees and codegrees. Then $G$ is CP parabolic if and only if it is recurrent for simple random walk.
\end{thm}

Let us now explain how these theorems follow from the more general statements given in \cite{he1999rigidity}. In that paper, He considers \textbf{disc patterns} of simple proper plane triangulations, which are like circle packings except that circles corresponding to adjacent vertices intersect at a specified angle, between $0$ and $\pi$, rather than being tangent. Given a proper plane network $G=(V,E)$ with  with locally finite dual and face set $F$, we can form a proper plane triangulation $T$ with vertex set $V \cup F$ by adding a vertex inside each face of $G$ and connecting this vertex to each vertex in the boundary of the face. It is easily verified that $T$ is simple if and only if $G$ is polyhedral. Moreover, a double circle packing of $T$, either in the plane or the disc, can now be obtained using Theorem 1.3 of \cite{he1999rigidity} by requiring that the angle between a pair of circles that are adjacent in $T$ is $0$ if the pair corresponds to an edge of $G$ and $\pi/2$ otherwise.  It is straightforward to check that conditions (C1) and (C2) of Theorem 1.3 of \cite{he1999rigidity} hold, and that the packing obtained this way is indeed the double circle packing of $G$. Thus, the existence statement of \cref{dblCP1} follows. The uniqueness of this packing, as formulated in \cref{dblCP1}, is the content of Theorems 1.1 and 1.2 of \cite{he1999rigidity}. Lastly, \cref{dblCP2} is a direct consequence of Theorem 1.3 of \cite{he1999rigidity} together with Theorems 2.6 and 8.1 of \cite{HeSc}.

Given a double circle packing $(P,P^\dagger)$ of a plane graph $G$ in a domain $D\subseteq \C$, we write $\diam_\C(A)$  for the Euclidean diameter of the set $\{z(v) : v\in A\}$, and write $d_\C(A,B)$ for the Euclidean distance between the sets $\{z(v) : v\in A\}$ and $\{z(v):v\in B\}$. We write $r(v)$ and $r(f)$ for the Euclidean radii of the circles $\Pp{v}$ and $\Pd{f}$. If $(P,P^\dagger)$ has carrier $\D$, we write $\sigma(v)$ for $1-|z(v)|$, which is the distance between $z(v)$ and the boundary of $\D$, and write $r_\bbH(v)$ for the hyperbolic radius of $\Pp{v}$. (Recall that Euclidean circles in $\D$ are also hyperbolic circles with different centres and radii; see e.g.\ \cite{anderson2006hyperbolic} for further background on hyperbolic geometry.)

We will also use
 the Ring Lemma of Rodin and Sullivan \cite{RS87}; see \cite{Hansen} and \cite{Ahar97} for quantitative versions. In \cref{subsec:ring} we formulate and prove a version of the Ring Lemma for double circle packings, \cref{Thm:DCPRing}.

\begin{thm}[The Ring Lemma]\label{T:Ring} There exists a sequence of positive constants $\langle k_m : m\geq 3\rangle$ such that for every circle packing $P$ of every simple triangulation $T$ in a domain $D \subseteq \C \cup \{\infty\}$, and every pair of adjacent vertices $u$ and $v$ of $T$ such that $P(v)$ does not contain $\infty$, the ratio of Euclidean radii
$ r(v)/r(u)$ is at most $k_{\deg(v)}$.
\end{thm}
An immediate corollary of the Ring Lemma is that, whenever $P$ is a circle packing in $\D$ of a CP hyperbolic proper plane triangulation $T$ and $v$ is a vertex of $T$, the hyperbolic radius of $P_v$ is bounded above by a constant $C=C(\deg(v))$.

\section{Connectivity of the FUSF}\label{Sec:mainproof}

In this section we prove \cref{T:planeWUSF} and show it easily implies \cref{T:mainthm}. We will write $\preceq$, $\succeq$ and $\asymp$ for inequalities that hold up to a positive multiplicative constant depending only on $\sup_{f\in F} \deg(f)$ and $\sup_{e\in E}c(e)^{-1}$.

\medskip

\textbf{Disclaimer:} The proof of \cref{T:planeWUSF} will also essentially yield a proof of the upper bound of \cref{T:Wtailcp,T:Ftailcp}. In particular, the inequality \eqref{eq:upperbound} is a direct analogue of this bound. Unfortunately, since \cref{T:mainthm,T:planeWUSF} require weaker assumptions on the graph $G$ than \cref{T:Wtailcp,T:Ftailcp} do, it will be technically convenient for the proofs of this section to use the circle packing of the triangulation obtained by adding a star inside each face of $G$ (sometimes referred to as the \emph{ball-bearing packing}), rather than the double circle packing of $G$: For \cref{T:mainthm,T:planeWUSF} it is better to use the ball-bearing packing than the double circle packing due to the way the Ring Lemma works, while the fact that the double circle packing defines a good embedding of $G$ in the sense of \cite{ABGN14} makes it better to use in \cref{T:Wtailcp,T:Ftailcp}. Nevertheless, we will want to use the bound analogous to \eqref{eq:upperbound} for double circle packing in \cref{Sec:lowerbound}, which holds for any polyhedral proper plane graph with bounded local geometry. The proof of this bound follows by a superficial modification of the proof of \eqref{eq:upperbound}, the details of which are omitted. The reader may wish to keep this in mind as they read this section, and observe that the proofs are easily generalised to the alternative setting once the Ring Lemma for double circle packings (\cref{Thm:DCPRing}) is established.

\subsection{Preliminaries}

We begin with some preliminary estimates that will be used in the proof. We first reduce the statement of \cref{T:planeWUSF} to the case where the graph is simple and has no \emph{peninsulas}.

Let $G$ be a bounded codegree proper plane network with locally finite dual $G^\dagger$. Recall that a \textbf{peninsula} of $G$ is a finite connected component of $G\setminus\{v\}$ for some vertex $v\in V$. %
In order to apply the theory of circle packing, we first reduce to the case in which $G$ is simple and does not contain a peninsula; this will ensure that the triangulation $T=T(G)$ formed by drawing a star inside each face of $G$ is simple. We will then use the circle packing of $T$ to analyse the WUSF of $G$. We first deal with the case that \emph{every} vertex of $G$
is in a peninsula. 

\medskip

%We begin by reducing to the case that $G$ is simple and does not contain a peninsula.

\begin{lem}\label{lem:peninsulas}
Let $G$ be a network such that every vertex $v$ of $G$ is contained in a peninsula of $G$. Then every essential spanning forest of $G$ is connected and one-ended. In particular, it follows that the WUSF of $G$ is connected and one-ended a.s.
 % $\F$ is connected and one-ended a.s.
\end{lem}

\begin{proof}
% Since every vertex of $G$ is contained in
Let $v_0$ be an arbitrary vertex of $G$, and for each $i\geq 1$ let $v_i$ be a vertex of $G$ such that $v_{i-1}$ is contained in a finite connected component of $G\setminus\{v_{i}\}$.
 Let $u$ be another vertex of $G$ and let $\gamma$ be a path from $v_0$ to $u$ in $G$. If $u$ is contained in an infinite connected component of $G\setminus\{v_i\}$, then $\gamma$ must pass through the vertex $v_i$. Since $\gamma$ is finite,  we deduce that there exists an integer $I$ such that $u$ is contained in a finite connected component of $G\setminus\{v_i\}$ for all $i\geq I$.
Thus, any infinite simple path from $u$ in $G$ must visit $v_i$ for all $i\geq I$. We deduce that every essential spanning forest of $G$ is connected and one-ended as claimed.
\end{proof}

Thus, to prove \cref{T:planeWUSF}, it suffices to consider the case that there is some vertex of $G$ that is not contained in a peninsula. In this case, let $G'$ be the simple plane network formed from $G$ by first splitting every edge $e$ of $G$ into a path of length $2$ with both edges given the weight $c(e)$, and then deleting every peninsula of the resulting network. The assumption that $G$ has bounded codegrees and edge resistances bounded above ensures that $G'$ does also (indeed, the maximum codegree of $G'$ is at most twice that of $G$).

% A graph is said to be $2$-connected if it does not contain a peninsula.

\begin{lem} \label{lem:nopeninsulas}
Let $G$ be a plane network and let $G'$ be as above. Then every component of the WUSF of $G$ is one-ended a.s.\ if and only if every component of the WUSF of $G'$ is one-ended a.s.
\end{lem}

\begin{proof}
Let $\F$ be a sample of $\WUSF_G$, and let $\F'$ be the essential spanning forest of $G'$ defined as follows. For each edge $e\in\F$ such that neither endpoint of $e$ is contained in a peninsula of $G$, let both of the edges of the path of length two corresponding to $e$ in $G'$ be included in $\F'$. For each edge $e\notin\F$ such that neither endpoint of $e$ is contained in a peninsula of $G$, choose uniformly and independently exactly one of the two edges of the path of length two corresponding to $e$ in $G'$ to be included in $\F'$. Then every component of $\F'$ is one-ended if and only if every component of $\F$ is one-ended, and it is easily verified that $\F'$ is distributed according to $\WUSF_{G'}$.
\end{proof}

Thus it suffices to consider the case that $G$ is simple and has no peninsulas. This assumption allows us to circle pack $G$ as follows. Let $T$ be the triangulation obtained by adding a vertex inside each face of $G$ and drawing an edge between this vertex and each vertex of the face it corresponds to. The assumption that $G$ is simple and does not contain a peninsula ensures that $T$ is a simple triangulation. We identify the vertices of $T$ with $V(G)\cup F(G)$, where $F(G)$ is the set of faces of $G$. Let $P=\{\Pp{v}: v \in V(G)\}\cup\{\Pp{f} : f \in F(G)\}$ be a circle packing of $T$ in either the plane or the unit disc. (Note that this is not the double circle packing of $G$, which is less useful for us at this stage since we are not assuming that $G$ has bounded degrees.)

We proceed with two geometric lemmas. For each $r'>r>0$ and $z$ in the carrier of $P$ let $V_{z}(r,r')$ be the set of vertices $v$ of $G$ such that either the intersection of the circle $\Pp{v}$ with the annulus $A_{z}(r,r')$ is non-empty or
the circle $\Pp{v}$ is contained in the ball $B_{z}(r)$ and there is a face $f$ incident to $v$ such that the intersection $\Pp{f} \cap A_{z}(r,r')$ is non-empty. We define the set $W_z(r):=V_z(r,r)$ similarly by using a circle instead of an annulus.

\begin{lemma}[Existence of a uniformly large number of disjoint annuli  around a point close to the boundary]\label{lem:uniformdisjointannuli}
Suppose that $T$ is CP hyperbolic so that the carrier of $P$ is $\D$. Then the following hold:
\begin{enumerate}
	% [leftmargin=]
\itemsep=0em
\item[$(1)$]
There exists a decreasing sequence  $\langle r_n \rangle_{n\geq 0}$ with $r_n \in (0,1/4)$
% \text{ and } \quad 1=\eps_1>\eps_2>\cdots\]
 such that for every $z\in\D$ with $|z|\geq 1-r_n$, the sets
\[ V_z(r_i,2r_i) \qquad 1\leq i \leq n \]
are disjoint.
\item[$(2)$] If $G$ has bounded degrees, then there exists a constant $C=C(\bM_G)$ such that we may take
$r_n = C^{-n}$ in the previous statement.
\end{enumerate}
\end{lemma}

\begin{proof}
We first prove item $(1)$.
	We construct the sequence recusively, letting $r_0=1/8$. Suppose that $\langle r_i \rangle_{i=0}^n$ satisfying the conclusion of the lemma have already been chosen, and
	consider the set of vertices
	\[ K_n=\Big\{ v \in V(G) : r(v) \geq \frac{r_n}{4} \text{ or } r(f)\geq \frac{r_n}{4} \text{ for some face $f$ incident to $v$}\Big\}.\]
Since the carrier of $P$ has finite area $K_n$ is a finite set.
	We define $r_{n+1}$ to be
	\[ r_{n+1} = \sup \big\{ r \leq r_{n}/4 \,:\, \Pp{v} \subseteq \overline{B_0(1-3r)} \text{ for all } v \in K_n\big\}\]
	which is positive since $K_n$ is finite.
	 We claim that $\langle r_i \rangle_{i=0}^{n+1}$ continues to satisfy the conclusion of the lemma.
That is, we claim that $V_z(r_{n+1},2r_{n+1}) \cap V_z(r_i,2r_i) =\emptyset$  for every $1\leq i\leq n$ and every $z\in \D$ with $|z| \geq 1- r_{n\asaf{+1}}$.

	Indeed, let $z$ be such that $|z|\geq 1-r_{n+1}$ and let $v \in V_z(r_{n+1},2r_{n+1})$. By definition of $V_z(r_{n+1},2r_{n+1})$, the intersection $\Pp{v}\cap B_z(2r_{n+1})$ is non-empty, and we deduce that $\Pp{v}$ is not contained in the closure of $B_0(1-3r_{n+1})$.
By definition of $r_{n+1}$, it follows that $v \notin K_n$ and hence that $r(v) \leq r_n/4$ and $r(f) \leq r_n/4$ for every face $f$ incident to $v$. Thus, since $r_{n+1}\leq r_n/4$,
	\[ \Pp{v}\subseteq B_z\left(2r_{n+1}+2r_n/4\right)\subseteq B_z\left(r_n\right)
	\text{
	 and
	  }
	 \Pp{f}\subseteq B_z\left(2r_{n+1}+2r_n/4\right)\subseteq B_z\left(r_n\right)\]
	 for every face $f$ incident to $v$.
	 It follows that
	   $V_z(r_{n+1},2r_{n+1}) \cap V_z(r_i,2r_i) =\emptyset$  for every $1\leq i\leq n$ as claimed.

% \medskip

To prove $(2)$, observe that, by the Ring Lemma (\cref{T:Ring}), there exists a constant $k=k(\bM)$ such that for every $C>1$, every  $z$ with $|z| \geq 1- C^{-m}$, and every $0\leq n\leq m$, every circle in $P$ that either has centre in $A_z(C^{-n},2C^{-n})$  or is tangent to some circle with centre in $A_{z}(C^{-n},2C^{-n})$ has radius at most $k C^{-n}$. Thus, this set of circles is contained in the ball $B_{z}((2+4k)C^{-n})$.
It follows that taking $C= 1/(4+8k)$ suffices. \qedhere
\end{proof}

Lastly, we estimate the energy of a random set of vertices that we will frequently use.

\begin{lemma}\label{lem: flow}  Let $z$ be a point in the carrier of $P$ (which may be either $\C$ or $\D$), let $U$ be a uniform random variable on the interval $[1,2]$  and, for each $r>0$, let $\mu_r$ be the law of the random set of vertices $W_z(Ur)=V_z(Ur,Ur)$. Then
$$\cE(\mu_r)\preceq 1$$
uniformly in $r>0$.
\end{lemma}

\begin{proof}
For a vertex $v$ of $G$ to be included in $W_z(Ur)$, the circle $\{z'\in\C: |z'-z|=Ur\}$
must either intersect the circle $\Pp{v}$ or intersect $\Pp{f}$ for some face $f$ incident to $v$. The union of $\Pp{v}$ and all the $\Pp{f}$ incident to $\Pp{v}$ is contained in the ball of radius $r(v) + 2\max_{f \sim v}r(f)$ around $z(v)$. Since the codegrees of $G$ are bounded, the Ring Lemma implies that $r(f) \preceq  r(v)$ for  all incident $v\in V$ and $f\in F$, and so
\begin{align}
	\mu_r(v \in W_z(Ur)) &\leq \frac{1}{r}\min\left(2r(v) + 4\max_{f \ni v}r(f),\,r\right) \preceq \frac{1}{r}\min\{r(v),r\}.\label{eq:energy1}
\end{align}
We claim that
\begin{equation}\label{eq:energy2}\sum_{v\in V_z(r,2r)} \min\{r(v),\, r\}^2 \leq 16r^2.\end{equation}
To see this, replace each circle of a vertex in $V_{z}(r,2r)$ that has radius larger than $r$ with a circle of radius $r$ that is contained in the original circle and intersects $B_z(2r)$: The circles in this new set still have disjoint interiors, are contained in the ball $B_z(4r)$, and have total area $\pi\sum_{v\in V_z(r,2r)} \min(r(v),\, r)^2$, yielding \eqref{eq:energy2}. The claim follows from  \eqref{eq:energy1} and \eqref{eq:energy2} by definition of the energy $\cE(\mu_r)$. \qedhere
\end{proof}

We are now ready to prove \cref{T:planeWUSF}.

\subsection{Proof of Theorem \ref{T:planeWUSF}}\label{subsec:exploration}
Let $G$ be a simple proper plane network with bounded codegrees and bounded edge resistances. By \cref{lem:nopeninsulas} we can assume that $G$ does not contain a peninsula. Let $\F$ be a sample of $\WUSF_G$ and given an edge $e=(x,y)$ let $\sA^e$ be the event that $x$ and $y$ are in distinct infinite connected components of $\F \setminus \{e\}$. It is clear that every component of $\F$ is one-ended a.s.\ if and only if
\begin{equation}\label{eq:mainthmgoal} \WUSF_G(e\in \F \, , \sA^e)= 0 \end{equation}
for every edge $e$ of $G$.% such that $\WUSF_G(\sA^e)>0$.

Consider the triangulation $T$ obtained from $G$ and its circle packing $P=\{\Pp{v}: v \in V(G)\}\cup\{\Pp{f} : f \in F(G)\}$, as described in the previous subsection. By applying a M\"obius transformation, we normalise $P$ by setting the centres $z(x)$ and $z(y)$ to be on the negative and positive real axes respectively, setting the circles $\Pp{x}$ and $\Pp{y}$ to have the origin as their tangency point and, in the parabolic case, fixing the scale by setting $z(y)-z(x)=1$.

Let $\eps>0$ be arbitrarily small.  If $T$ is CP hyperbolic, let $V_\eps$ be the set of vertices of $G$ with $|z(v)|\leq 1 -\eps$. Otherwise, $T$ is CP parabolic and we define $V_\eps$ to be the set of of vertices of $G$ with $|z(v)|\leq \eps^{-1}$. We also denote by $E_\eps$ the set of edges that have both endpoints in $V_\eps$.

%%% ASAF says: it's unclear to me why we defined two events B_eps^e and C_eps^e; it seems to me that only C_eps^e is relevant. So I'm changing it.
Let $\sB_\eps^e$ be the event that every component of $\F\setminus \{e\}$ intersects $V\setminus V_\eps$.
 %and that there is no path between $x$ and $y$ using only edges of $(\F \cap E_\eps) \setminus\{e\}$. Note that $\sA^e = \cap_{\eps>0} \sB^e_\eps$.
On the event $\sB_\eps^e$, we define $\eta^x$ to be the rightmost path in $\F\setminus\{e\}$ from  $x$ to $V\setminus V_\eps$ when looking at $x$ from $y$, and $\eta^y$ to be the leftmost path in $\F\setminus\{e\}$ from $y$ to $V\setminus V_\eps$ when looking at $y$ from $x$. Note that the paths $\eta_x$ and $\eta_y$ are not necessarily disjoint. Nonetheless, concatenating the reversal of $\eta^x$ with $e$ and $\eta^y$ separates $V_\eps$ into two sets of vertices, $\cL$ and $\cR$, which are to the left and right of $e$ (when viewed from $x$ to $y$) respectively. See Figure \ref{fig:exploration} for an illustration of the case when $\eta_x$ and $\eta_y$ are disjoint (when they are not, $\cR$ is a ``bubble'' separated from $V \setminus V_\eps$).

%We also let $\fL$ and $\fR$ be the corresponding open regions of the ball $B_0(1-\eps)$\asaff{, in the CP hyperbolic case, or the ball $B_0(\eps^{-1})$, in the CP parabolic case}.
On the event $\sB^e_\eps$, let $K$ be the set of edges that are either incident to a vertex in $\cL$ or belong to the path $\eta_x \cup \eta_y$, setting $K=E$ off of this event. Note that the edges of $K$ do not touch the vertices of $\cR$. The condition that $\eta^x$ and $\eta^y$ are the rightmost and leftmost paths to $V\setminus V_\eps$ from $x$ and $y$ is equivalent to the condition that $K$ does not contain any open path from $x$ to $V\setminus V_\eps$ other than $\eta^x$, and does not contain any open path from $y$ to $V\setminus V_\eps$ other than $\eta^y$. It follows from this characterisation that  $K$ is a local set for $\F$. (Indeed, we note that $K$ can be explored algorithmically, without querying the status of any edge in $E\setminus K$, by performing a right-directed depth-first search of $x$'s component in $\F$ and a left-directed depth-first search of $y$'s component in $\F$, stopping each search when it first leaves $V_\eps$.)

\begin{figure}[t]
\centering
\includegraphics[width=0.44\textwidth]{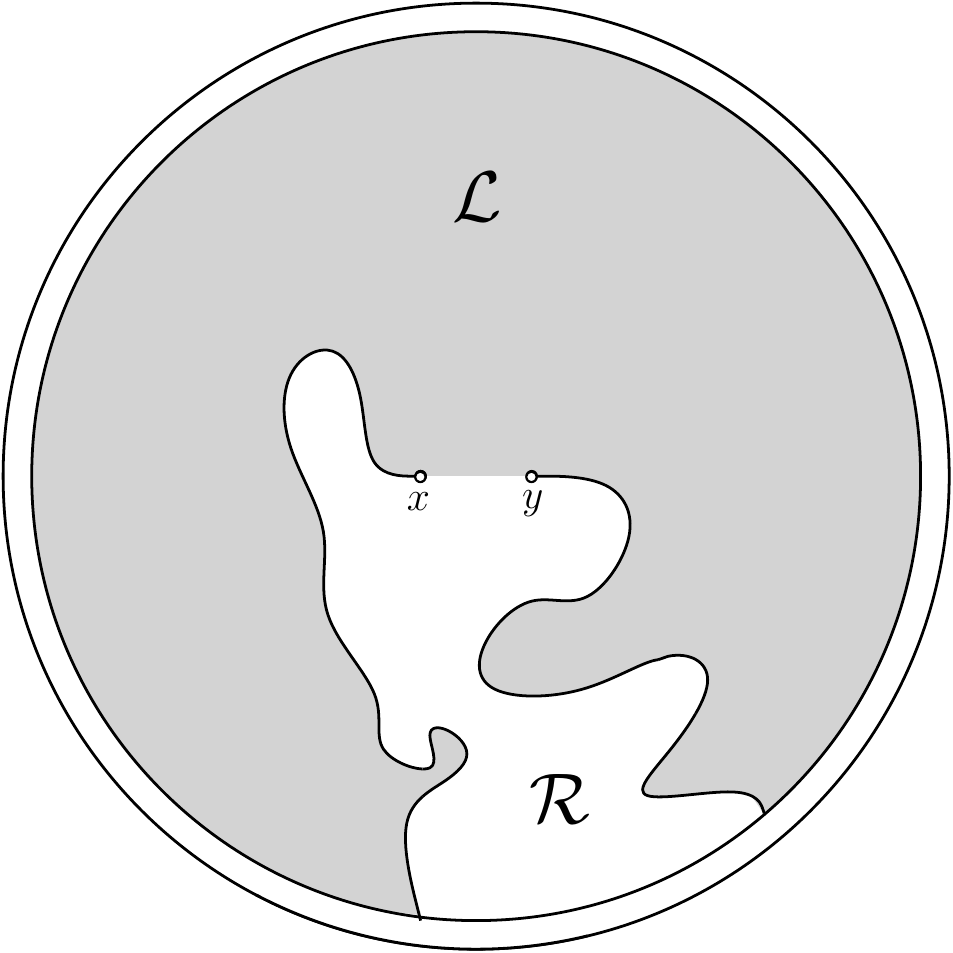} \qquad \includegraphics[width=0.44\textwidth]{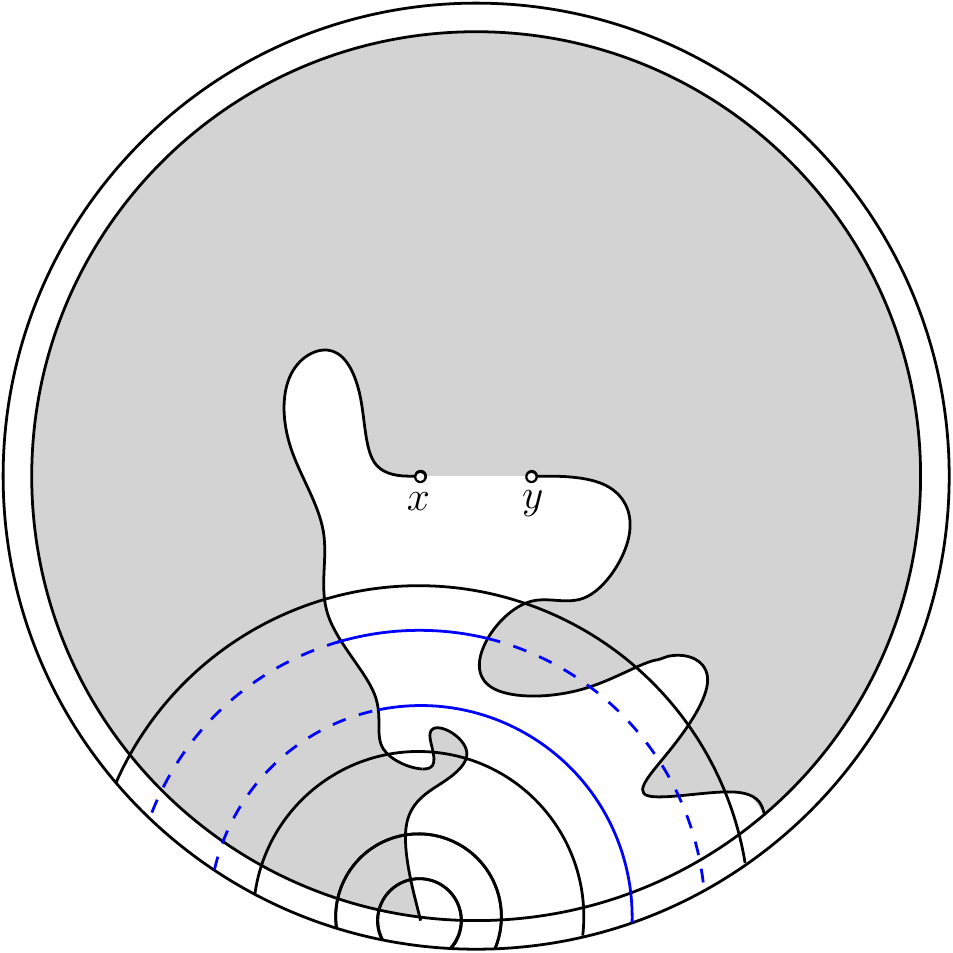}
\caption{Illustration of the proof of \cref{T:planeWUSF} in the case that $T$ is CP hyperbolic. Left: On the event $\sA^e_\eps$, the paths $\eta^x$ and $\eta^y$ split $V_\eps$ into two pieces, $\cL$ and $\cR$. Right: We define a random set containing a path (solid blue) from $\eta^x$ to $\eta^y\cup\{\infty\}$ in $G\setminus K_c$ using a random circle (dashed blue). Here we see two examples, one in which the path ends at $\eta^y$, and the other in which the path ends at the boundary (i.e., at infinity).}\label{fig:exploration}
\end{figure}

Let $\sA^e_\eps$ denote the event that $\sB_\eps^e$ occurs and that $K$ does not contain an open path from $x$ to $y$, or, equivalently, that $\eta_x$ and $\eta_y$ are disjoint. Note that $\sA_\eps^e$ is measurable with respect to $\cF_K$ (as defined in \cref{subsec:Markov}), and that $\sA^e = \cap_{\eps>0} \sA^e_\eps$. Thus, 
$$ \wusf_G(e \in \F \, , \sA^e) \leq \wusf_G(e \in \F \mid \sA_\eps^e) = \E \big [ \wusf_G(e \in \F  \, |\, \cF_K, \, \sA_\eps^e) \big ] \, .$$
% Let $K_c$ denote the set of closed edges of $\F$ in $K$ and let $K_o$ denote the set of open edges of $\F$ in $K$.
By the strong spatial Markov property (\cref{prop:Markov}), conditioned on $\cF_K$ and the event $\sA_\eps^e$, the law of $\F$ is equal to the union of $K_o$ with a sample of the WUSF  of the network $(G-K_c)/K_o$ obtained from $G$ by deleting all the revealed closed edges and contracting all the revealed open edges. In particular, by Kirchhoff's Effective Resistance Formula (\cref{Thm:Kirchhoff}),
\begin{equation}\label{eq:wusfkirchhoffexploration} \wusf_G\big(e\in\F \, |\, \cF_K,\,\sA_\eps^e\big) \leq c(e)\,\ReffW\big(\eta^x \leftrightarrow \eta^y;\; G- K_c\big)
 \end{equation}
%$\sA^e \subseteq \sC_\eps^e$
Since the edge $e$ was arbitrary, to prove \eqref{eq:mainthmgoal} (and hence \cref{T:planeWUSF}) it suffices to prove that there is a upper bound on the effective resistance appearing in \eqref{eq:wusfkirchhoffexploration} that tends to zero as $\eps\to0$ uniformly in $\cF_K$. We perform this analysis now according to whether $T$ is CP hyperbolic or parabolic.

\medskip

%If $T$ is CP hyperbolic, Let $v^x$ be the endpoint of the path $\eta^x$ and let $z_0=z(v^x)$. Otherwise, $T$ is CP parabolic and we let $z_0=0$.
% For each $r',r>0$, and $z \in \C$, we define the ball
% \[ B_{z}(r) = \{ z' \in \C : |z-z'| < r \}\] and the annulus
% \[ A_{z}(r,r')=\{ z' \in \C : r \leq |z'-z| \leq r' \}. \]

\begin{proof}[Proof of \cref{T:planeWUSF}, hyperbolic case.] Suppose that $T$ is CP hyperbolic, let $v^x$ be the endpoint of the path $\eta^x$ and let $z_0=z(v^x)$. 
On the event $\sA_\eps^e$, for each $1-|z_0|\leq r\leq 1/4$, we claim that
 the set $W_{z_0}(r)$, as defined in \cref{lem: flow}, contains a path in $G$ from $\eta^x$ to $\eta^y \cup \{\infty\}$ that is contained in $\cR \cup \eta^x \cup \eta^y$, and is therefore a path in $G\setminus K_c$.

Indeed, consider the arc $\fA'(r)=\{z\in \overline{\D} : |z-z_0|=r\}$,  parameterised in the clockwise direction. Let $\fA(r)$ be the subarc of $\fA'(r)$ beginning at the last time that $\fA'(r)$ intersects a circle corresponding to a vertex in the trace of $\eta^x$, and ending at the first time after this time that $\fA'(r)$ intersects either $\partial \D$ or a circle corresponding to a vertex in the trace of $\eta^y$ (see \cref{fig:exploration}).
Thus, on the event $\sA^e_\eps$, the set of vertices of $T$ whose corresponding circles in $P$ are intersected by $\fA(r)$ contains a path in $T$ from $\eta^x$ to $\eta^y\cup\{\infty\}$, for every $1-|z_0|\leq r\leq 1/4$. (Indeed, for Lebesgue a.e.\ $1-|z_0| \leq r \leq 1/4$, the arc $\fA(r)$ is not tangent to any circle in $P$, and in this case the set is precisely the trace of a simple path in $T$.) To obtain a path in $G$ rather than $T$, we divert the path counterclockwise around each face of $G$. That is, whenever the path passes from a vertex $u$ of $G$ to a face $f$ of $G$ and then to a vertex $v$ of $G$, we replace this section of the path with the list of vertices of $G$ incident to $f$ that are between $u$ and $v$ in the counterclockwise order. This construction shows that the subgraph of $G \setminus K_c$ induced by the set $W_{z_0}(r)$ contains a path from $\eta^x$ to $\eta^y\cup\{\infty\}$, as claimed.

Let $\langle r_n \rangle_{n\geq0}$ be as in \cref{lem:uniformdisjointannuli} and let $n(\eps)$ be the maximum $n$ such that $\eps<r_n$. By \cref{lem:uniformdisjointannuli} the measures $\mu_{r_i}$ are supported on sets that are contained in the disjoint sets $V_z(r_i,2r_i)$. Thus, by \cref{lem:randomsets} and \cref{lem: flow} we have
\begin{align*} \ReffW\Big(\eta^x \leftrightarrow \eta^y \cup \{\infty\};\, G \setminus K_c\Big)  \preceq \cE\left(\frac{1}{n(\eps)}\sum_{i=1}^{n(\eps)}\mu_{r_i}\right)=\frac{1}{n(\eps)^2}\sum_{i=1}^{n(\eps)} \cE(\mu_{r_i}) \preceq \frac{1}{n(\eps)}  \end{align*}
and hence, by symmetry,
\[ \ReffW(\eta^y \leftrightarrow \eta^x \cup \{\infty\};\, G \setminus K_c) \preceq \frac{1}{n(\eps)}.  \]
Applying \cref{Lem:triangle} and \eqref{eq:wusfkirchhoffexploration}, we have
\begin{equation} \wusf(e\in\F \, |\, \cF_K,\,\mathscr{B}_\eps^e) \preceq \frac{c(e)}{n(\eps)}, \label{eq:neps}\end{equation}
which by \cref{lem:uniformdisjointannuli} converges to zero as $\eps \to 0$, completing the proof of \cref{T:planeWUSF} in the case that $T$ is CP hyperbolic. If $G$ has bounded degrees, then combining \eqref{eq:neps} with \cref{lem:uniformdisjointannuli}  implies that there exists a positive constant $C=C(\bM)$ such that
\begin{equation}\label{eq:upperbound}\wusf(e\in\F \, |\, \cF_K,\,\mathscr{B}_\eps^e) \leq C\frac{c(e)}{\log(1/\eps)}\end{equation}
 for all $\eps \leq 1/2$, which is a direct analogue of the upper bound of \cref{T:Wtailcp}.
\end{proof}

\begin{proof}[Proof of \cref{T:planeWUSF}, parabolic case.]

Suppose that $T$ is CP parabolic.
Let $r_1=1$ and define $\langle r_n \rangle_{n\geq1}$ recursively by
\[ r_n = 1+\inf\{r\geq r_{n-1} : V_0(r,2r) \cap V_0(r_{n-1},2r_{n-1})=\emptyset\}.\]
Since $V_0(r_{n-1},2r_{n-1})$ is finite, this infimum is finite.
 By definition, the sets $V_0(r_i,2r_i)$ are disjoint. A similar analysis to the hyperbolic case shows that, on the event $\sA_\eps^e$, for each $r\geq 1$,
 the set $W_0(r)$ contains a path in $G$ from $\eta^x$ to $\eta^y$ that is contained in $\cR \cup \eta^x \cup \eta^y$, and is therefore a path in $G\setminus K_c$.  For each $\eps>0$, let $n(\eps)$ be the maximal $n$ such that $r_n \leq \eps^{-1}$. Then on the event $\sA_\eps^e$, by \cref{lem:randomsets} and \cref{lem: flow},
\begin{align*} \ReffF(\eta^x \leftrightarrow \eta^y;\; G \setminus K_c) &\preceq \cE\left(\frac{1}{n(\eps)}\sum_{i=1}^{n(\eps)}\mu_{r_i}\right) \leq \frac{1}{n(\eps)^2}\sum_{i=1}^{n(\eps)} \cE(\mu_{r_i}) \preceq \frac{1}{n(\eps)}.
% \\ &\leq \frac{1}{n(\eps)}18(1+2k_{G})^2\sup_{e\in E}c(e)^{-1}.
\end{align*}
Thus, by \eqref{eq:wusfkirchhoffexploration},
\begin{equation}\label{eq:parabolicKirchhoff} \WUSF_G(e\in\F \, |\, \cF_K,\,\mathscr{B}_\eps^e) \preceq \frac{c(e)}{n(\eps)}.\end{equation}
The right hand side converges to zero as $\eps \to 0$, completing the proof of \cref{T:planeWUSF}.
\end{proof}

\begin{remark} Since the random sets used in the CP parabolic case above are always finite, they can also be used to bound free effective resistances. Therefore, by repeating the proof above with the FUSF in place of the WUSF, we deduce that every component of the FUSF of $G$ is one-ended almost surely if $T$ is CP parabolic.
Since the FUSF stochastically dominates the WUSF, and an essential spanning forest with one-ended components does not contain any strict subgraphs that are also essential spanning forests, we deduce that if $T$ is CP parabolic, then the FUSF and WUSF of $G$ coincide. In particular, using \cite[Theorem 7.3]{BLPS}, we obtain the following. \end{remark}

\begin{thm}\label{T:parabolicHD}
Let $T$ be a CP parabolic proper plane triangulation.
Then $T$ does not admit non-constant harmonic functions of finite Dirichlet energy.
\end{thm}

\subsection{The FUSF is connected.}

\begin{proof}[Proof of \cref{T:mainthm}]
First suppose that $G^\dagger$ is locally finite. Since $G$ has bounded degrees and bounded conductances, $G^\dagger$ has bounded codegrees and bounded resistances. Thus, \cref{T:planeWUSF} implies that every component of the WUSF on $G^\dagger$ is one-ended a.s.\ and consequently, by \cref{thm:duality}, that the FUSF  of $G$ is connected a.s.

Now suppose that $G$ does not have locally finite dual. In this case, we form a plane network $(G',c')$ from $G$ by adding edges to triangulate the infinite faces of $G$ while keeping the degrees bounded. We enumerate these additional edges $\langle e_i \rangle_{i\geq1}$ and define conductances \[c'(e_i)=2^{-i-1}\ReffF\big(e_i^- \leftrightarrow e_i^+ ; G\big)^{-1}.\]
Since $G'$ has bounded degrees, bounded conductances and locally finite dual, its FUSF is connected a.s.
By Kirchhoff's Effective Resistance Formula (\cref{Thm:Kirchhoff}), Rayleigh monotonicity and the union bound, the probability that the FUSF of $G'$
 contains any of the additional edges $e_i$ is at most
\[\sum_{i\geq0}c'(e_i)\ReffF\big(e_i^- \leftrightarrow e_i^+ ; G'\big) \leq \sum_{i\geq0}c'(e_i)\ReffF\big(e_i^- \leftrightarrow e_i^+ ; G\big) \leq 1/2.\]
In particular, there is a positive probability that none of the additional edges are contained in the FUSF of $G'$.
The conditional distribution of the FUSF of $G'$ on this event is $\FUSF_G$ by the spatial Markov property, and it follows that the FUSF of $G$ is connected a.s.
\end{proof}

\section{Critical Exponents}\label{Sec:lowerbound}

\subsection{The Ring Lemma for double circle packings}\label{subsec:ring}

In this section we extend the Ring Lemma to double circle packings. While unsurprising, we were unable to find such an extension in the literature.

\begin{thm}[Ring Lemma]\label{Thm:DCPRing} There exists a family of positive constants $\langle k_{n,m} : n\geq 3, m\geq 3\rangle$ such that if $(P,P^\dagger)$ is a double circle packing in $\C \cup \{\infty\}$ of a polyhedral plane graph $G$ and $v$ is a vertex of $G$, then for every $f\in F$ incident to $v$
such that $\Pp{v}$ does not contain $\infty$, then
\[r(v)/r(f) \leq k_{\deg(v),\max_{g \perp v}\deg(g)} \]
where $g \perp v$ means that the face $g$ is incident to the vertex $v$.
\end{thm}

As for triangulations, the Ring Lemma immediately implies that whenever a polyhedral, CP hyperbolic proper plane network $G$ with bounded degrees and codegrees is circle packed in $\D$, the hyperbolic radii of the discs in $P\cup P^\dagger$ are bounded above by a constant $C$ depending only on the maximal degree and codegree.

\begin{proof}[Proof of \cref{Thm:DCPRing}]

Let $n$ be the degree of $v$ and let $m$ be the maximum degree of the faces incident to $v$ in $G$.
We may assume that $r(v)=1$.  Note that for each two distinct discs $\Pd{f},\Pd{{}f'}\in P^\dagger$ that are not tangent, there is at most one disc $\Pp{u}\in P$ that intersects both $\Pd{f}$ and $\Pd{{}f'}$, while if $\Pd{f}$ and $\Pd{{}f'}$ are tangent, there exist exactly two discs in $P$ that intersect both $\Pd{f}$ and $\Pd{{}f'}$.
For each two faces $f$ and $f'$ incident to $v$, the complement $\partial \Pp{v}\setminus \left(\Pd{{}f_1} \cup \Pd{{}f_2}\right)$ is either a single arc (if $\Pd{f}$ and $\Pd{{}f'}$ are tangent), or is equal to the union of two arcs (if $\Pd{f}$ and $\Pd{{}f'}$ are not tangent).

We claim that there exists an  function
$\psi_{m}(\cdot,\cdot): (0,\infty)^2 \to (0,2\pi],$
increasing in both coordinates,
 such that if $f$ and $f'$ are two distinct faces of $G$ incident to $v$, then each of the (one or two) arcs forming the complement $\partial \Pp{v}\setminus (\Pd{{}f_1} \cup \Pd{{}f_2})$  have length at least $\psi_{m}(r(f),r(f'))$. Indeed, let $r(f)$ and $r(f')$ be fixed, and suppose that one of the arcs forming the complement $\partial \Pp{v} \setminus \Pd{f} \cup \Pd{{}f'}$ is extremely small, with length $\eps$. Let the primal circles incident to $f$ be enumerated $v_1,\ldots,v_{\deg(f)}$, where $v_1=v$, $\Pp{{}v_2}$ is the primal circle that is tangent to $\Pp{v}$ and
intersects $\Pd{f}$ on the same side as the small arc, $\Pp{{}v_3}$ is the next primal circle  that is tangent to $\Pp{v_2}$ and intersects $\Pd{f}$, and so on. Since $\eps$ is small, $\Pp{{}v_2}$ must also be very small, as it does not intersect $\Pd{{}f'}$. Similarly, if $\eps$ is sufficiently small, $\Pp{{}v_3}$ must also be small, since it also does not intersect $\Pd{{}f'}$. See \cref{fig:ring}. Applying this argument recursively, we see that, if $\eps$ is sufficiently small, then the circles $\Pp{{}v_2},\ldots,\Pp{{}v_{\deg(f)}}$ are collectively too small to contain $\partial \Pd{f} \setminus \Pp{v}$ in their union, a contradiction. We write $\psi_m(r(f),r(f'))$ for the minimal $\eps$ that is not ruled out as impossible by this argument.

\begin{figure}[t]
\centering
\includegraphics[height=0.3\textwidth]{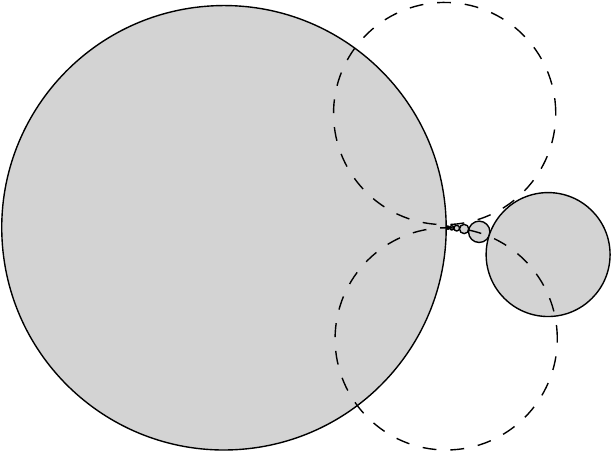}\hspace{1cm}\includegraphics[height=0.3\textwidth]{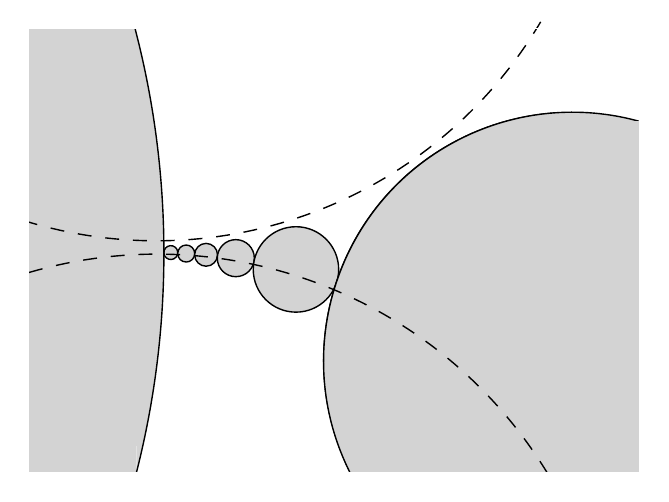}
\caption{Proof of the Double Ring Lemma. If two dual circles are close but do not touch, there must be many  primal circles  contained in the crevasse between them. This forces the two dual circles to each have large degree. The right-hand figure is a magnification of the left-hand figure.
 }
\label{fig:ring}
\end{figure}

Let the faces incident to the vertex $v$ be indexed  in clockwise order $f_1,\ldots,f_n$, where $n =\deg(v)$ and $f_1$ has maximal radius among the faces incident to $v$. For each face $f$ incident to $v$, the arc $\partial \Pp{v} \cap \Pd{f}$ has length $2\tan^{-1}(r(f))$.
 Since $\partial \Pp{v} = \cup_{i=1}^n (\partial \Pp{v} \cap \Pd{f_i})$, we deduce that $r(f_1)$ is bounded below by $\tan(\pi/n)$.
 By definition of $\psi_m$, we have that $r(f_k)$ satisfies
\begin{equation}\label{eq:ringproof}\psi_{m}\left(\tan(\pi/n),r(f_k)\right) \leq \psi_{m}\left(r(f_1),r(f_k)\right) \leq \sum_{i=2}^{k-1}2\tan^{-1}\left(r(f_i)\right)\end{equation}
for all $3 \leq k\leq n$. For each such $k$, \eqref{eq:ringproof} yields an implicit upper bound on $r(f_k)$ which converges to zero as $r(f_2)$ converges to zero.
This in turn yields a uniform lower bound on $r(f_2)$: if $r(f_2)$ were sufficiently small, the bound \eqref{eq:ringproof} would imply that $\sum_{k=2}^n 2\tan^{-1}(r(f_k))$ would be less than $\pi$, a contradiction (since we have trivially that $2\tan^{-1}(r(f_1)) \leq \pi$). We obtain uniform lower bounds on $r(f_k)$ for each $3\leq k \leq n$ by repeating the above argument inductively.
\end{proof}

\subsection{Good embeddings of planar graphs}
% \cref{subsec:estimates}

If $G$ is a plane graph and $(P,P^\dagger)$ is a double circle packing of $G$ in a domain $D \subseteq \C$, then
drawing straight lines between the centres of the circles in $P$ yields a proper embedding of $G$ in $D$ in which every edge is a straight line. We call such an embedding a \textbf{proper straight-line embedding of $G$} in $D$. Following \cite{ABGN14}, a proper straight-line embedding of a graph $G$ in  a domain $D \subseteq \C$ is said to be \textbf{$\eta$-good} if the following conditions are satisfied:
\begin{enumerate}\item \textbf{(No near-flat, flat, or reflexive angles.)} All internal angles of every face in the drawing are at most $\pi-\eta$. In particular, every face is convex.
\item \textbf{(Adjacent edges have comparable lengths.)} For every pair of edges $e,e'$ sharing a common endpoint, the ratio of the lengths of the straight lines corresponding to $e$ and $e'$ in the drawing is at most $\eta^{-1}$.
\end{enumerate}
% We note that together, these conditions also imply the following condition.
% \begin{enumerate}
% \item[(3)] \textbf{(No near-zero angles.)}  All internal angles of every face in the drawing are at most $\eta\sin(\eta/2)$.
% \end{enumerate}
The Ring Lemma (\cref{Thm:DCPRing}) has the following immediate corollary.
\begin{corollary}\label{Cor:DCPgood}
The proper straight-line embedding given by any double circle packing  of a plane graph $G$ with bounded degrees and bounded codegrees \tomm{in a domain $D \subseteq \C$} is $\eta$-good for some positive $\eta=\eta(\bM_G)$.
\end{corollary}

\asaf{\tom{\noindent We remark that, in contrast, the embedding of $G$ obtained by circle packing the triangulation $T(G)$ formed by drawing a star inside each face of $G$ (and then erasing these added vertices), as done in \cref{Sec:mainproof}, does not necessarily yield a good embedding \tomm{of $G$.}}}

For the remainder of this section and in \cref{subsec:estimates,subsec:lowerboundproof,Subsec:dualexponents}, $G$ will be a fixed transient, polyhedral proper plane network with bounded codegrees and bounded local geometry, and $(P,P^\dagger)$ will be a double circle packing of $G$ in $\D$. We will write $\preceq$, $\succeq$ and $\asymp$ to denote inequalities or equalities that hold up to positive multiplicative constants depending only upon $\bM$.
We will also fix an edge $e=(x,y)$ of $G$ and, by applying a M\"obius transformation if necessary, normalise $(P,P^\dagger)$ by setting the centres $z(x)$ and $z(y)$ to be on the negative real axis and positive real axis respectively and setting the circles $\Pp{x}$ and $\Pp{y}$ to have the origin as their tangency point.

In \cite{ABGN14} and \cite{Chelkak}, several estimates are established that allow one to compare the random walk on a good embedding of a planar graph with Brownian motion.
The following estimate, proven for general good embeddings of proper plane graphs in \cite[Theorems 1.4 and 1.5]{AHNR15}, is of central importance to the proofs of the lower bounds in \cref{T:Wtailcp} and \cref{Thm:Wareacp}.
Recall that $\sigma(v)$ is defined to be $1-|z(v)|$.

\begin{thm}[Diffusive Time Estimate]\label{Thm:Time}
% Let
% $G$ be a proper plane network with bounded codegrees and bounded local geometry, and let  $(P,P^\dagger)$ be a double circle packing of $G$ in $\D$.
There exists a constant $C_{1}=C_{1}(\bM)\geq 1$ such that the following holds. For each vertex $v$ of $G$,
% Let $v$ be a vertex of $G$,
 let $\langle X_n \rangle_{n\geq0}$ be a random walk on $G$ started at $v$, let $r(v) \leq r \leq C_{1}^{-1}\sigma(v)$ and let $T_r$ be the first time $n$ that $|z(X_n)-z(v)|\geq r$. Then
\[\bE_v\sum_{n=0}^{T_r}r(X_n)^2 \asymp r^2.\]
\end{thm}

Murugan [personal communication] has shown that the constant $C_1$ above can in fact be taken to be $1$.
% a stronger version of \cref{Thm:Time}.

\begin{thm}[Cone Estimate]\label{Thm:Cone}
Let $\eta=\eta(\bM)>0$ be the constant from \cref{Cor:DCPgood}. There exists a positive constant $q_1=q_1(\bM)$ such that the following holds. For each vertex $v$ of $G$, let $\langle X_n \rangle_{n\geq0}$ be a random walk on $G$ started at $v$, let $0\leq r \leq \sigma(v)$, and let $T_r$ be the first time $n$ that $|z(X_n)-z(v)|\geq r$. Then for any interval $I \subset \R/(2\pi \Z)$ of length at least $\pi - \eta$ we have
\[ \bP_v\left( \arg\left(z(X_{T_r}) - z(v)\right) \in I \right) \geq q_1   . \]
\end{thm}

% \note{Give the constant a different name}
We also apply the following result of Benjamini and Schramm \cite[Lemma 5.3]{BS96a}; their proof was given for simple triangulations but extends immediately to our setting.

\begin{thm}[Convergence Estimate \cite{BS96a}]\label{thm:convergenceestimate}
For every vertex $v$ of $G$, we have that
% $G$ be a transient, polyhedral, proper plane network with bounded codegrees and bounded local geometry, and let  $(P,P^\dagger)$ be a double circle packing of $G$ in $\D$.
% There exists a positive constant $C$ depending only on the maximum degree, codegree, edge resistance and edge conductance of $G$ such that
\[\bP_v\Bigl( \left|z(X_n)-z(v)\right| \geq t \sigma(v) \text{ for some }n\geq0\Bigr) \preceq \frac{1}{\log t}.\]
% $1-C/\log(1/\delta)$ to remain disjoint from each other by \cite[Lemma 5.3]{BS96a}, where $C$ is a positive constant depending only on the maximum degree, codegree, edge resistance and edge conductance of $G$.
\end{thm}

We remark that \asaf{the logarithmic decay in} \cref{thm:convergenceestimate} is not sharp. Sharp \asaf{polynomial} estimates of the same quantity have been obtained by Chelkak \cite[Corollary 7.9]{Chelkak}.

% Let $I \subset \Z$ be an interval, and  let $\gamma : I \to V $ be a simple cycle in $G$. The union of the images under $z$ of the edges in $\gamma$ is a simple closed curve in $\C$, and we let $\Omega_\gamma \subset \C$ be the domain bounded by this curve. For each subinterval $J \subset I$, we define $J_\C \subset \C$ to be the union of the images under $z$ of each edge traversed by $\gamma$ during the interval $J$, together with the half-edges

% \begin{thm}[Harmonic measure estimate {\cite[Corollary 7.9]{Chelkak}}]\label{Thm:toolbox} There exist positive constants $\alpha,\beta$ and $C$ such that for every vertex $v$ such that $z(v) \in \Omega_\gamma$
% \[C^{-1}\bP_{z(v)}(B_{\tau_{\C\setminus \Omega_\gamma}} \in I_\C)^\alpha \leq \bP_{v}(X_{\tau_\gamma} \in I) \leq C\bP_{z(v)}(B_{\tau_{\C\setminus \Omega_\gamma}} \in I_\C)^\beta\]
% % Let $\gamma$ be a simple cycle in $G$, let $a,b$ be vertices
% \end{thm}

\subsection{Preliminary estimates}\label{subsec:estimates}

For each vertex $v$ of $G$, let $a_{\bbH}(v)$ denote the hyperbolic area of $\Pp{v}$ (recall that by the uniqueness of \cref{dblCP1} this quantity does not depend on the choice of packing). Consider a M\"obius transformation of the unit disc taking $z(v)$ to the origin. Then by the Ring Lemma (\cref{Thm:DCPRing}), the Euclidean radius of $v$ must be bounded by a constant $c=c(\bM)<1$ and we deduce that the hyperbolic radii of the discs in $P$ are bounded from above. Thus,
\begin{equation}\label{eq:hypeuclcomp}\ah(v) \asymp r_\bbH(v)^2 \asymp \sigma(v)^{-2}r(v)^2\end{equation}
for every vertex $v$ of $G$. 

Recall that given a set of vertices $B$ we write $\tau_B$ for the first time the random walk visits $B$, letting $\tau_B=\infty$  if no such time exists. In the following section, we will wish to estimate sums of the form
\begin{equation}
\sum_{u \in A}\ah(u)\bP_{u}(\tau_B<\infty)\label{eq:sumform}
\end{equation}
where $A$ and $B$ are subsets of $V$.
We begin with a useful preliminary estimate. For $r\in (0,1]$ and a vertex $v$ let $\fH(v,r)$ be the half-plane containing $z(v)$ whose boundary is the unique straight line with distance $r\sigma(v)$ to $z(v)$ that is orthogonal to the line connecting $z(v)$ to the origin.

\begin{lem}[Remain in half-plane estimate] \label{lem:ballandcone} For any $r\in(0,1]$ there exists a positive constant $q_2=q_2(\bM,r)$, that is increasing in $r$, such that for any vertex $v$ the probability of the random walk starting at $v$ to remain in $\fH(v,r)$ forever is at least $q_2$. That is,
\begin{equation}\label{eq:coneline}
 \bP_v\big( z(X_n) \in \fH(v,r) \text{ for all } n \big ) \geq q_2.
 \end{equation}
\end{lem}

\begin{proof} 
This is an estimate similar to ones made in \cite{ABGN14}. We will use \cref{Thm:Cone} to ``push'' the walker deep inside the half plane $\fH(v,r)$ (the number of times we apply \cref{Thm:Cone} will only depend on $r$) and once it is close enough to the boundary \cref{thm:convergenceestimate} will ensures that it stays close forever with positive probability. We now make this precise. For each vertex $u$, let $\fC(u)$ be the cone
\[ \fC(u) =  \big  \{ z \in \C : |\arg z - \arg z(u)| \leq \pi/2 - \eta/2 \big \} \, . \]

Define numbers $\langle \rho_i\rangle_{i\geq0}$ and stopping times $\langle T_i\rangle_{i\geq0}$ by putting $T_0=0$ and $\rho_0=r \sigma(v)$ and recursively setting for $i \geq 1$
\[T_i = \min\left\{n \geq 0: |z(X_n) - z(X_{T_{i-1}})|\geq \rho_{i-1}/2\right\} \, ,\]
and setting $\rho_i$ to be the distance between $z(X_{T_i})$ and $\partial \fH(v,r) \cup \partial \D$.
Since the closed ball of radius $\rho_{i-1}/2$ around $z(X_{T_{i-1}})$ is contained in $\D$, it has only finitely many vertices in it and therefore $T_i < \infty$ almost surely for all $i \geq 1$.

Denote by $\sA_i$ the event that $z(X_{T_{i}}) \in \fC(X_{T_{i-1}})$,
or in words, that the random walk starting from $X_{T_{i-1}}$ exits the ball of radius $\rho_{i-1}/2$ around $z(X_{T_{i-1}})$ inside the cone that is parallel to $\fC(v)$ and is centered at $z(X_{T_{i-1}})$. By basic trigonometry this event implies that
\begin{equation}\label{eq:pushclosetobdry}\sigma(X_{T_{i}}) \leq \sigma(X_{T_{i-1}}) - \rho_{i-1} \sin(\eta/2)/2 \, ,\end{equation}
and
\begin{equation}\label{eq:pushfarfromhalfplane} d_\C(X_{T_{i}},\partial \fH(v,r))\geq  d_\C(X_{T_{i-1}},\partial \fH(v,r)) + \rho_{i-1} \sin(\eta/2)/2. \end{equation}
In particular, the distance between $z(X_{T_i})$ and $\partial \D$ is decreasing in $i$ and the distance between $z(X_{T_i})$ and $\partial \fH(v,r)$ is increasing in $i$.

By  \cref{thm:convergenceestimate} there exists $t_0=t_0(\bM)\geq 1$ such that we have
\[\bP_u\left( \left|z(X_n)-z(u)\right| \geq t_0 \sigma(u) \text{ for some }n\geq0\right) \leq {1 \over 2} \, ,\]
for every vertex $u$. Set $\delta=\min \{r (2t_0)^{-1}, \sin(\eta/2)\}$ and put $n_0 = 2\lceil 1/\delta^2 \rceil$. We first note that by \eqref{eq:pushfarfromhalfplane} the event $\cap_{i=0}^{n_0} \sA_i$ implies that $z(X_n) \in \fH(v,r)$ for $0 \leq n \leq T_{n_0}$. Next we have that if $X_{T_i}$ has distance at least $\delta\sigma(v)$ from $\partial \fH(v,r) \cup \partial \D$, then $\rho_i\sin(\eta/2)\geq \delta^2\sigma(v)$. Therefore, if the event $\cap_{i=0}^{n_0} \sA_i$ occurs, then $X_{T_{n_0}}$ has distance at most $\delta\sigma(v)$ from $\partial \D$ (otherwise the contradictory assertion that $\sigma(X_{T_{n_0}}) \leq 0$ follows by \eqref{eq:pushclosetobdry}) and has distance at least $r\sigma(v)$ from $\partial \fH(v,r)$ by \eqref{eq:pushfarfromhalfplane}. That is,
$$ \sigma(X_{T_{n_0}}) \leq r(2t_0)^{-1} \sigma(v) \qquad \hbox{and} \qquad d_\C(X_{T_{n_0}},\partial \fH(v,r)) \geq r \sigma(v) \, .$$
Therefore, by the definition of $t_0$ and the strong Markov property, conditioned on the event $\cap_{i=0}^{n_0} \sA_i$ we have that $|z(X_n) - z(X_{T_{n_0}})|\leq r \sigma(v)/2$ for all $n \geq T_{n_0}$ with probability at least $1/2$. Thus, we deduce that
$$ \bP_v(z(X_n) \in \fH(v,r) \hbox{ for all } n \geq T_{n_0} \mid \cap_{i=0}^{n_0} \sA_i) \geq {1 \over 2} \, .$$
\cref{Thm:Cone} and the strong Markov property imply that $\bP_v(\sA_i \mid X_{T_{i-1}}) \geq q_1$, and hence that
$$ \bP_v\Bigl(\cap_{i=0}^{n_0} \sA_i\Bigl) \geq q_1^{n_0} \, ,$$
which concludes our proof by taking $q_2 = q_1^{n_0}/2$.
\end{proof}

Let $p,r\in (0,1]$.
We say that a set $A \subset V$ is \textbf{$(p,r)$-escapable} if for every vertex $v\in A$, the random walk started at $v$ has probability at least $p$ of not returning to $A$ after first leaving the set of vertices whose corresponding discs in $P$ have centres contained in the  Euclidean ball of  radius $r\sigma(v)$ about $z(v)$.
To avoid trivialities, we also declare the empty set to be $(p,r)$-escapable for all $p$ and $r$. \cref{lem:ballandcone} leads to the following corollary.

\begin{corollary}\label{lem:escape-1}
Let $r\in (0,1]$. There exist positive constants $\delta_2=\delta_2(\bM,r)$ and $p_1=p_1(\bM,r)$, both increasing in $r$, such that the set
\[\{v\in V: (1-\delta_2)\eps \leq \sigma(v) \leq \eps\}  \]
is $(p_1,r)$-escapable for every $\eps>0$.
\end{corollary}

\begin{proof}
Let $\eta=\eta(\bM)$ be the constant appearing in \cref{Cor:DCPgood}, and set $$\delta_2 = \min\bigl\{r \sin(\eta/2)/2,\, 1/4\bigl\}\, .$$
Let $\eps>0$ be arbitrary, let $v \in \{v\in V: (1-\delta_2)\eps \leq \sigma(v) \leq \eps\}$, and
let $T$ be the stopping time
\[T=\min\left\{n \geq 0: |z(X_n) - z(v)|\geq r\sigma(v)\right\}.\]
We apply \cref{Thm:Cone} with $r \sigma(v)$ to get
\begin{equation}
\label{eq:space0}
\bP_v\left(\sigma(X_{T}) \leq \sigma(v) -  r\sin(\eta/2)\sigma(v) \right) \geq q_1(\bM) > 0.\end{equation}
By definition of $\delta_2$, on the event appearing in the left-hand side of \eqref{eq:space0}, we have $\sigma(X_T) \leq \eps(1-2\delta_2)$. Therefore, the half-plane $\fH(X_T,\delta_2)$, defined above \cref{lem:ballandcone}, is disjoint from the ball $\{z \in \mathbb{C} : |z| \leq 1-(1-\delta_2)\eps\}$. Thus, the claim follows from \eqref{eq:space0} and \cref{lem:ballandcone} by taking $p=q_1(\bM)q_2(\bM,\delta_2)$, where $q_2$ from \cref{lem:ballandcone}.
\end{proof}

\begin{lem}\label{lem:escape1}
Let $A\subset V$ be $(p,r)$-escapable for some $p\in(0,1)$ and $r\in (0,C_1^{-1}/2]$, where $C_1=C_1(\bM)$ is the constant appearing in \cref{Thm:Time}. Then for every vertex $u\in V$,
\begin{equation}\bE_{u}\Bigg[\sum_{n\geq0}\ah(X_n)\mathbbm{1}(X_n \in A)\Bigg] \preceq \frac{r^2}{p}.\label{eq:escapetime}
 \end{equation}
\end{lem}

\begin{proof}

Define the sequences of stopping times $\langle T_i^-\rangle_{i\geq0}$ and $\langle T_i^+ \rangle_{i\geq0}$ by letting $T_0^-=\tau_A$ and recursively letting
\[T_i^+=\min\left\{n \geq T_i^-: |z(X_n) - z(X_{T_i^-})|\geq r\sigma(X_{T_i^-})\right\}\]
and
$T_i^-=\min\{n\geq T_{i-1}^+: X_n \in A\}$.
Then
\begin{align*}
\bE_{u}\Bigg[\sum_{n\geq0}\ah(X_n)\mathbbm{1}(X_n \in A)\Bigg] \leq \sum_{i\geq0}\bE_{u}\Bigg[\mathbbm{1}(T_i^-<\infty)\sum_{n=T_i^-}^{T_i^+-1}\ah(X_n)\Bigg].
\end{align*}
Since $\sigma(X_n)\geq (1-r)\sigma(X_{T_{i}^-})$ for all $T_i^- \leq n < T_i^+$, the Diffusive Time Estimate (\cref{Thm:Time}), the strong Markov property, and \eqref{eq:hypeuclcomp} imply that
\begin{multline*}
\bE_u\Bigg[\sum_{n=T_i^-}^{T_i^+-1}\ah(X_n) \,\Bigg|\, T_i^-<\infty,\, X_{T_i^-} \Bigg]
 \preceq (1-r)^{-2}\sigma(X_{T_i^-})^{-2}\bE_u\Bigg[\sum_{n=T_i^-}^{T_i^+}r(X_n)^2 \,\Bigg|\, T_i^-<\infty,\, X_{T_i^-} \Bigg]\preceq r^2.
\end{multline*}
Meanwhile, since $A$ is $(p,r)$-escapable, we have that $\bP_u(T_i^-<\infty)\leq (1-p)^{i}$.
Combining these estimates yields the desired inequality \eqref{eq:escapetime}.
\end{proof}

We say that a set $A\subset V$ is $C$-\textbf{short-lived} if \begin{equation}\label{def:shortlived}\bE_{u}\Bigg[\sum_{n\geq0}\ah(X_n)\mathbbm{1}(X_n \in A)\Bigg] \leq C
 \end{equation}
for every $u \in V$.
 Note that if $A$ is a $C$-short-lived set of vertices, then every subset of A is also $C$-short-lived. Also note that \cref{lem:escape1} states that escapable sets are short-lived.

\begin{lem}\label{Lem:esacape2}
Let $A$ and $B$ be two sets of vertices in $G$, and suppose that $A$ is finite and $C$-short-lived for some $C>0$. Then
\[\sum_{v\in A} \ah(v)\bP_v(\tau_B<\infty) \preceq C\CeffF(A\leftrightarrow B). \]
\end{lem}

\begin{proof}
 Let the stopping times $\tau_i$ be defined recursively by setting $\tau_0=\tau_{A}$ and $\tau_{i+1}=\min\{t>\tau_i: X_t \in A\}$, so that $\tau_i$ is the $i$th time the random walk $\langle X_n \rangle_{n\geq0}$ visits $A$. As usual if $A$ is not hit by the random walk we set the corresponding stopping time to $\infty$. Then
\begin{align*}
\sum_{v\in A} \ah(v)\bP_{v}(\tau_{B} < \infty) &\leq \sum_{v\in A} \sum_{i\geq0}\ah(v)\bP_{v}(\text{$B$ hit between time $\tau_i$ and $\tau_{i+1}$})\\
&=\sum_{v\in A} \sum_{u\in A} \sum_{i\geq0} \ah(v)\bP_{v}(\tau_i<\infty,\, X_{\tau_i}=u)\bP_{u}(\tau_B < \tau_{A}^+).
% \\
\end{align*}
Reversing time gives that $\bP_{v}(\tau_i<\infty,\, X_{\tau_i}=u) = {\deg(u) \over \deg(v)} \bP_{u}(\tau_i<\infty,\, X_{\tau_i}=v)$ and since the degrees are bounded we get
\begin{align*}
\sum_{v\in A} \ah(v)\bP_{v}(\tau_{B} < \infty) \preceq \sum_{v\in A} \sum_{u\in A} \sum_{i\geq0} \ah(v)\bP_{u}(\tau_i<\infty,\, X_{\tau_i}=v)\bP_{u}(\tau_B < \tau_{A}^+).
\end{align*}
By exchanging the order of summation and using our assumption on $A$ we obtain that
\begin{align*}
\sum_{v\in A} \ah(v)\bP_{v}(\tau_{B} < \infty) &\preceq
\sum_{u\in A} \bE_{u}\Bigg[\sum_{n\geq0}\ah(X_n)\mathbbm{1}(X_n \in A)\Bigg] \bP_{u}(\tau_B < \tau_{A}^+)
\\
& \preceq C\sum_{u\in A} \bP_{u}(\tau_B < \tau_{A}^+).
 % \leq MC_6 \sum_{u\in A}c(u)\bP_{u}(\tau_B<\tau_{A}^+).
\end{align*}
To conclude, let $\langle V_j \rangle_{j\geq1}$ be an exhaustion of $V$ and let $G_j$ be the subgraph of $G$ induced by $V_j$, with conductances inherited from $G$, and observe that
\begin{align*}
 \sum_{u\in A} c(u)\bP_{u}^G(\tau_{B} < \tau_{A}^+) &= \lim_{j\to\infty}\sum_{u\in A} c(u)\bP_{u}^{G}(\tau_{B} < \min\{\tau_{A}^+,\tau_{ V\setminus V_j}\})\\ &\leq  \lim_{j\to\infty}\sum_{u\in A} c(u)\bP_{u}^{G_j}(\tau_{B} < \tau_{A}^+) = \CeffF(A \leftrightarrow B) \, ,
\end{align*}
concluding the proof since $c(u)$ is bounded away from $0$.
\end{proof}

Recall that $\diam_\C(A)$ and $d_\C(A,B)$ denote the Euclidean diameter of $\{z(v):v\in A\}$ and the Euclidean distance between $\{z(v):v\in A\}$ and $\{z(v):v\in B\}$ respectively.

\begin{lem}\label{lem:conddiam}
Let $A$ and $B$ be disjoint sets of vertices in $G$. Then
\[\CeffF(A\leftrightarrow B) \preceq
\frac{\diam_\C(A)^2}{\min\left\{\diam_\C(A),\,d_\C(A,B)\right\}^2},
\]
with the convention that the right-hand side is $1$ if $\diam_\C(A)=0$.
\end{lem}

\begin{proof}
Let $D=\diam_\C(A)$. If $D=0$ then $A$ is a single vertex $v$ and $\CeffF(A\leftrightarrow B) \leq c(v) \preceq 1$, so assume not.  Recall the \textbf{extremal length} characterisation of the free effective conductance \cite[Exercise 9.42]{LP:book}:  For each function  $\ell:E\to[0,\infty)$  assigning a non-negative length to every edge $e$ of $G$, let $d_\ell$ be the shortest path pseudometric on $G$ induced by $\ell$. Then
\[\CeffF(A \leftrightarrow B) = \inf\left\{\frac{\sum_{e\in E}c(e)\ell(e)^2}{d_\ell(A,B)^2} : \ell: E \to [0,\infty),\, d_l(A,B) >0  \right\}. \]
Let $W$ be the set of vertices $v$ of $G$ whose corresponding circles intersect the $D$-neighbourhood of $z(A)$ in $\C$.
 Define lengths by setting $\ell(e)$ to be
\[\ell(e)=\left\{
% \begin{array}{ll}
\barraydebug{ll}
\min\left\{|z(e^-)-z(e^+)|,\,D\right\} & \text{ if $e$ has an endpoint in $W$}\\
0 & \text{ otherwise}.
% \end{array}
\earraydebug
 \right.\]
 Then
 \begin{equation}\label{eq:condiam1}
 d_{\ell}(A,B) \geq \min\left\{D,\, d_\C(A,B))\right\}>0
 \end{equation}
while, since $|z(e^-)-z(e^+)|=r(e^-)+r(e^+)$,
\begin{align}\label{eq:condiam3}
\sum_{e}c(e)\ell(e)^2 & \preceq \sum_{v\in W}\sum_{v'\sim v} \min\left\{r(v)+r(v'),\, D\right\}^2 \preceq \sum_{v\in W}\min\left\{r(v),\, D\right\}^2,
\end{align}
where the Ring Lemma (\cref{Thm:DCPRing}) is used in the second inequality.
 As in the proof of \cref{lem: flow}, consider replacing each circle corresponding to a vertex in $W$ that has radius larger than $D$ with a circle of radius $D$ that is contained in the original circle and intersects the $D$-neighbourhood of $z(A)$. This yields a set of circles contained in the $3D$-neighbourhood of $z(A)$. Comparing the total area of this set of circles with that of the $3D$-neighbourhood of $z(A)$ yields that
\begin{equation}\label{eq:condiam2}
 \sum_{v\in W}\min\left\{r(v),\,D\right\}^2\preceq D^2.
\end{equation}
We conclude by combining \eqref{eq:condiam1}, \eqref{eq:condiam3} and \eqref{eq:condiam2}.
\end{proof}

\subsection{Wilson's Algorithm}

\textbf{Wilson's algorithm rooted at infinity} \cite{Wilson96,BLPS} is a powerful method of sampling the WUSF of an infinite, transient graph by joining together loop-erased random walks. We now give a very brief description of the algorithm. See \cite{LP:book} for a detailed exposition.
Let $G$ be a transient network. Let $\gamma$ be a path in $G$ that visits each vertex of $G$ at most finitely many times. The \textbf{loop-erasure} is formed by erasing cycles from $\gamma$ chronologically as they are created. (The loop-erasure of a random walk path is referred to as \textbf{loop-erased random walk} and was first studied by Lawler \cite{lawler1980self}.)
Let $\{v_j : j \in \mathbb{N} \}$ be an enumeration of the vertices of $G$. Let $\F_0=\emptyset$ and define a sequence of forests in $G$ as follows:
\begin{enumerate}
\item Given $\F_i$, start an independent random walk from $v_{i+1}$. Stop this random walk if it hits the set of vertices already included in $\F_i$, running it forever otherwise.
\item Form the loop-erasure of this random walk path and let $\F_{i+1}$ be the union of $\F_i$ with this loop-erased path.
\end{enumerate}
Then the forest $\F =\bigcup_{i\geq0}\F_i$ is a sample of the WUSF of $G$ \cite[Theorem 5.1]{BLPS}.

\subsection{Proof of Theorems \ref{T:Ftailcp}, \ref{T:Wtailcp} and \ref{Thm:Wareacp}}\label{subsec:lowerboundproof}

Recall that $e=(x,y)$ is a fixed edge. We write  $\succeq_e$ to denote a lower bound that holds up to a positive multiplicative constant that depends only on $\bM$ and $r_\bbH(x)$, and that is increasing in $r_\bbH(x)$.

Let $(P,P^\dagger)$ be the double circle packing of $G$ in $\D$ normalized so that $z(x)$ and $z(y)$ lie on the negative and positive axes  respectively and so that the tangency point of $P(x)$ and $P(y)$ is the origin. By the Ring Lemma (\cref{Thm:DCPRing}) there exists a constant $c=c(\bM)<1$ such that for any $\eps>0$ and any vertex $v$ with $|z(v)|\geq 1-\eps$ we must have that $r(v) \leq c \eps$. It follows that there exists a constant $s=s(\bM)\in (0,1/2]$ such that for every $\eps>0$, the set
\[ \left\{v \in V : z(v) \in A_0(1-\eps,1-s\eps)\right\}\]
disconnects $e$ from $\infty$ in $G$. An equivalent formulation that we will use is that if $(v,u)$ is an edge in the graph and $0<\eps_1 < \eps_2$, then
\begin{equation}\label{eq:ringlem} \sigma(v) \in [\eps_1, \eps_2] \Longrightarrow \sigma(u) \in [s\eps_1, s^{-1}\eps_2] \, .\end{equation}
% Fix such an $s\leq 1/2$ and,
For each $\eps>0$, we define
\[W_\eps :=\left\{v \in V : z(v)\in  A_0(1-\eps,1-s^5\eps)\right\}.\]
Recall the definition of a $C$-short-lived set from \eqref{def:shortlived}. We trivially have that unions of boundedly many short-lived sets are short-lived with a larger constant. In particular, letting $s=s(\bM)$  be as above and letting $\delta=\delta_2(\bM,C_1^{-1}(\bM))$, where $C_1$ is the constant appearing in \cref{Thm:Time} and $\delta_2$ is the constant from \cref{lem:escape-1}, we have that
\[ W_\eps \subseteq \bigcup_{m=0}^{\lceil 5\log_{1-\delta}(s) \rceil}\bigl\{v \in V: (1-\delta)^{m+1}\eps \leq  \sigma(v)\leq (1-\delta)^m\eps\bigr\}  \]
for every $\eps>0$. Applying \cref{lem:escape-1,lem:escape1} immediately yields the following.

\begin{corollary}\label{lem:escape0}
There exist a constant $C_2=C_2(\bM)>0$ such that the set $W_\eps$ is $C_2$-short-lived for every $\eps>0$.
\end{corollary}

Let $\F$ be the WUSF of $G$ and let $\eps\in(0,1/2)$ be arbitrarily small. The proofs in this section will be based upon the study of the random variable
\begin{equation}\label{eq:defzeps} Z_\eps = \sum_{v \in W_\eps}\ah(v)\mathbbm{1}\bigl(v \in \past_\F(e)\bigr)\, ,\end{equation}
and its moments. Intuitively, we think of the process $\langle Z_{2^{-n}} \rangle_{n\geq 1}$ as behaving similarly to a critical branching process, so that in particular we expect the random variable $Z_\eps$ to be distributed similarly to the number of particles at generation $\log(1/\eps)$ in a critical branching process. In particular, we expect its first moment to be of order $1$ and its second moment to be of order $\log(1/\eps)$. This is what we prove below. We begin by bounding its first moment.

\begin{lem}\label{lem:firstmomentupper}
$\E[Z_\eps]\preceq 1$ for all $\eps>0$.
\end{lem}

\begin{proof}
If we generate $\F$ using Wilson's algorithm,
 starting with the vertex $v$, then $v$ is in $\past_\F(e)$ if and only if the loop-erased random walk from $v$ passes through $e=(x,y)$. In particular $\P(v \in \past_\F(e)) \leq \bP_v(\tau_{x}<\infty)$. Thus, the claim follows immediately from
\cref{lem:escape0,Lem:esacape2}, taking $A=W_\eps\setminus \{x\}$ and $B=\{x\}$, since $\CeffF(x\leftrightarrow V\setminus \{x\})\leq c(x)\preceq 1$.
\end{proof}

\begin{lem}  There exists a positive constant $\delta_5 =\delta_5(\bM,r_\bbH(x))$, increasing in $r_\bbH(x)$, such that  $\E[Z_\eps] \succeq_e 1$ for all  $\eps\leq \delta_5$. \label{lem:firstmoment}
\end{lem}

\begin{proof}

As in the previous lemma, $v$ is in $\past_\F(e)$ if and only if the loop-erased random walk from $v$ passes through $e=(x,y)$.
In particular, we obtain the lower bound
\[\P(v \in \past_\F(e)) \geq
\bP_v\left(
% \begin{array}{l}
\tau_{x}<\infty,\, X_{\tau_{x}+1}=y,\, \langle X_n \rangle_{n\geq \tau_x+1} \text{ disjoint from } \langle X_n \rangle_{n=0}^{\tau_x}
% \end{array}
\right)
\]
and hence, decomposing according to the value of $\tau_{x}$,
\begin{align*}\E[Z_\eps]
&\geq \sum_{v \in W_\eps}\sum_{m\geq1}\ah(v)\bP_v\left(
% \begin{array}{l}
\tau_x=m,\, X_{m+1}=y,\, \langle X_n \rangle_{n\geq m} \text{ disjoint from } \langle X_n \rangle_{n=0}^{m-1}
\right). \end{align*}
Letting $\langle Y_n\rangle_{n\geq0}$ be a random walk started at $x$ independent of $\langle X_n \rangle_{n\geq0}$ and reversing time yields that
\begin{align*}
\E[Z_\eps] \succeq \sum_{v \in W_\eps}\sum_{m\geq1}\ah(v)\bP_x\left(
X_m=v,\, \sC_m
\right)
= \bE_x\Bigg[\sum_{m\geq0}\ah(X_m)\mathbbm{1}\left(
% \begin{array}{l}
X_m\in W_\eps,\,
\sC_m
% Y_1=y,\, \langle X_n \rangle_{n =0}^m \text{ disjoint from } \langle Y_n \rangle_{n\geq 0}
% \end{array}
\right)\Bigg],
\end{align*}
where $\sC_m$ is the event
\[\sC_m = \left\{Y_1=y,\, \text{ and } \langle X_n \rangle_{n =0}^m \text{ disjoint from } \langle Y_n \rangle_{n\geq 0}\right\}.\]
(Note that on the event $\sC_m$ the walk $\langle X_n\rangle_{n=0}^m$ does not return to $x$ after time $0$.)

Let $\tau_1$ be the first time that the random walk $\langle X_m \rangle_{m\geq0}$ visits $\{v \in V : s^3\eps \leq \sigma(v) \leq s^2 \eps\}$, which is finite a.s.\ since this set separates $x$ from $\infty$ by definition of $s$. Let $\tau_2$ be the first time $m$ after $\tau_1$ that $|z(X_m)-z(X_{\tau_1})|\geq C_1^{-1}s^3\eps$, where $C_1 =C_1(\bM)\geq 1$ is the constant from \cref{Thm:Time}. By the triangle inequality
$$ s^3 \eps - C_1^{-1}s^3 \eps \leq \sigma(X_{\tau_2-1}) \leq s^2\eps + C_1^{-1}s^3 \eps \, ,$$
and hence by \eqref{eq:ringlem} we get that
$$ s^4 \eps - C_1^{-1}s^4 \eps \leq \sigma(X_{\tau_2}) \leq s \eps + C_1^{-1} s^2 \eps \, .$$
We may assume that $s$ is sufficiently small  that  $X_m \in W_\eps$ for all $\tau_1 \leq m \leq \tau_2$. Therefore,
\begin{eqnarray*}
	\E[Z_\eps] &\succeq&
	\bE_x\Bigg[\sum_{m = \tau_1}^{\tau_2}\ah(X_m)\mathbbm{1}\left(\sC_{m} \right)\Bigg] \geq \bE_x\Bigg[\sum_{m = \tau_1}^{\tau_2}\ah(X_m)\mathbbm{1}\left(\sC_{\tau_2} \right)\Bigg] \\
	&\geq& \bE_x\Bigg[\sum_{m = \tau_1}^{\tau_2}\ah(X_m)\mathbbm{1}\left(d_\C(X_{\tau_1},\{Y_n:n\geq0\})\geq 2\eps/s \textrm{ and } \sC_{\tau_2} \right)\Bigg] \, ,
	% &\geq \bE_x\left[\bE_x\left[\sum_{m= \tau_1}^{\tau_2} \ah(X_m) \bigmid X_{\tau_1}\right]\mathbbm{1}\left(d(z(X_{\tau_1}),\{Y_n : n\geq 0\}) > \eps, \sC_{\tau_1} \right)\right].
\end{eqnarray*}
where the second inequality follows since the events $\sC_m$ are decreasing and third inequality is trivial. We have that
$$ |z(X_{\tau_2-1}) - z(X_{\tau_1})| \leq C_1^{-1} s^3 \eps \, ,$$
and since $X_{\tau_1-1} \in W_\eps$ we have
$$ |z(X_{\tau_2}) - z(X_{\tau_2-1})| \leq \eps/s \, ,$$
by the definition of $s$. By the triangle inequality $|z(X_{\tau_2}) - z(X_{\tau_1})|\leq 2\eps/s$ since $s \leq 1/2$. We deduce that the events $\sC_{\tau_1}\cap\{d_\C(X_{\tau_1},\{Y_n:n\geq0\})\geq 2\eps/s\}$ and $\sC_{\tau_2}\cap\{d_\C(X_{\tau_1},\{Y_n:n\geq0\})\geq 2\eps/s\}$ are equal. Therefore,
\begin{align*}\E[Z_\eps]&\succeq \bE_x\left[\sum_{m= \tau_1}^{\tau_2} \ah(X_m)\mathbbm{1}\left(d_\C(X_{\tau_1},\{Y_n : n\geq 0\}) \geq 2\eps/s, \sC_{\tau_1} \right)\right]\\
&=\bE_x\left[\bE_x\left[\sum_{m= \tau_1}^{\tau_2} \ah(X_m) \bigmid X_{\tau_1}\right]\mathbbm{1}\left(d_\C(X_{\tau_1},\{Y_n : n\geq 0\}) \geq 2\eps/s, \sC_{\tau_1} \right)\right].
\end{align*}
Applying  \eqref{eq:hypeuclcomp} and the Diffusive Time Estimate (\cref{Thm:Time}), we obtain that
\begin{align*}
	\bE_x\left[\sum_{m= \tau_1}^{\tau_2} \ah(X_m) \bigmid X_{\tau_1}\right]
	&\asymp
	\eps^{-2}\bE_x\left[\sum_{m= \tau_1}^{\tau_2} r(X_m)^2 \bigmid X_{\tau_1}\right]
	\asymp 1,
\end{align*}
and hence
\begin{align}
	 \E[Z_\eps] &\succeq \bP_{x}\Bigl(d_\C(X_{\tau_1},\{Y_n : n\geq 0\}) \geq 2\eps/s, \sC_{\tau_1}\Bigr).
	\label{eq:firstmoment1}
\end{align}

We now set
$$ \delta_5 = {s(r(x) + r(y)) \over 4 } \, ,$$
and note that as usual by Ring Lemma $\delta_5 \asymp r_{\bbH}(x)$. Since the degrees are bounded by $\bM$, the probability that $Y_1=y$ is at least $\bM^{-1}$ and we may assume this indeed occurred. We now apply \cref{lem:ballandcone} with $r_1={r(x) \over 2 \sigma(x)}$ (note that by the Ring Lemma $\sigma(x) \succeq 1$ and hence $r_1 \succeq r_{\bbH(x)}$) and establish that with probability at least $q_2$ the walker $\{X_n\}_{n \geq 0}$ is confined to the half plane $\fH(x,r_1)$ (which is just $\{ (x,y) : x \leq -r(x)/2 \}$). Similarly, we put $r_2={r(y) \over 2 \sigma(y)}$ and by \cref{lem:ballandcone} we get that with probability at least $q_2$ the walker $\{Y_n\}_{n \geq 1}$ starting from $y$ is confined to the halfplace $\fH(y,r_2)$ (which is $\{ (x,y) : x \geq r(y)/2\}$). If these two events occur, then the distance between $z(X_n)$ and $z(Y_m)$ is at least $(r(x)+r(y))/2$ for any $n\geq 0$ and $m\geq 1$. By our choice of $\delta_5$ we deduce that for any $\eps \leq \delta_5$
\[\bP_{x}\Bigl(d_\C\bigl(X_{\tau_1},\{Y_n : n\geq 0\}\bigr) \geq 2\eps/s, \sC_{\tau_1}\Bigr) \succeq \bM^{-1} q_2\bigl(\bM,r_\bbH(x)\bigr)^2 \succeq_e 1 \, ,\]
concluding the proof. \qedhere

\end{proof}

We now wish to estimate the second moment of $Z_\eps$. Doing so amounts to estimating the conditional first moment of $Z_\eps$ given that some specified vertex of $W_\eps$ is contained in the past of $e$. 
We will show that this conditional first moment is $O(\log(1/\eps))$ by decomposing the annulus $W_\eps$ into $O(\log(1/\eps))$ many pieces, the sets $U_\eps^m$ below, and then proving that each such piece contributes at most a constant to the conditional first moment.  Using the bounds established in \cref{subsec:estimates} will reduce the problem of bounding the contribution to the conditional first moment of each piece to geometric estimates that can be proven by elementary trigonometry.

We now begin to carry out this program.
For each $\eps>0$ and $m\in \Z \setminus \{0\}$, we define
\[U^m_\eps = \left\{z \in \D: 1-\eps \leq |z| \leq 1- s^3\eps \text{ and } \sgn(m)\frac{4^{|m|}}{5^{|m|}}\pi\leq \arg z \leq \sgn(m)\frac{4^{|m|-1}}{5^{|m|-1}}\pi   \right\}\]
and define $U^m_\eps(\theta)$ to be the rotated set $e^{i\theta}U^m_\eps$.

\begin{lem}\label{lem:trig} There exist universal constants $\delta_3,\delta_4>0$ and $k<\infty$ such that
\[d_\C\left(U^m_\eps(\theta), \{re^{i\theta}: r \geq0\}\right) \geq (2+\delta_3)\diam_\C\left(U^m_\eps(\theta)\right)\] for all $\theta\in [-\pi,\pi]$, $\eps \leq \delta_4$ and $|m| \leq \log_{5/4}(1/\eps)-k$.
\end{lem}
The constants here are more important than usual since we will later need to estimate  the difference $\frac{1}{2}d_\C\left(U^m_\eps(\theta), \{re^{i\theta}: r \geq0\}\right) -  \diam_\C\left(U^m_\eps(\theta)\right)$.

\begin{figure}
\centering
\includegraphics[height=0.325\textwidth]{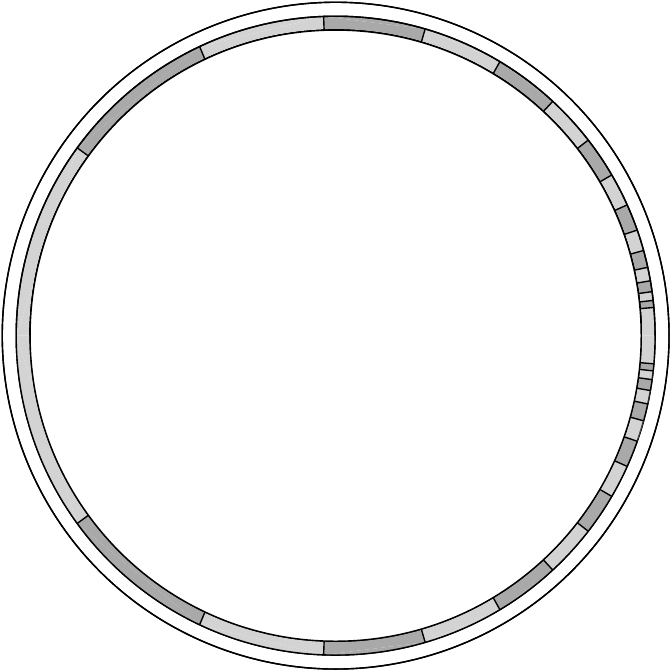}
\hspace{1cm}
\includegraphics[height=0.325\textwidth]{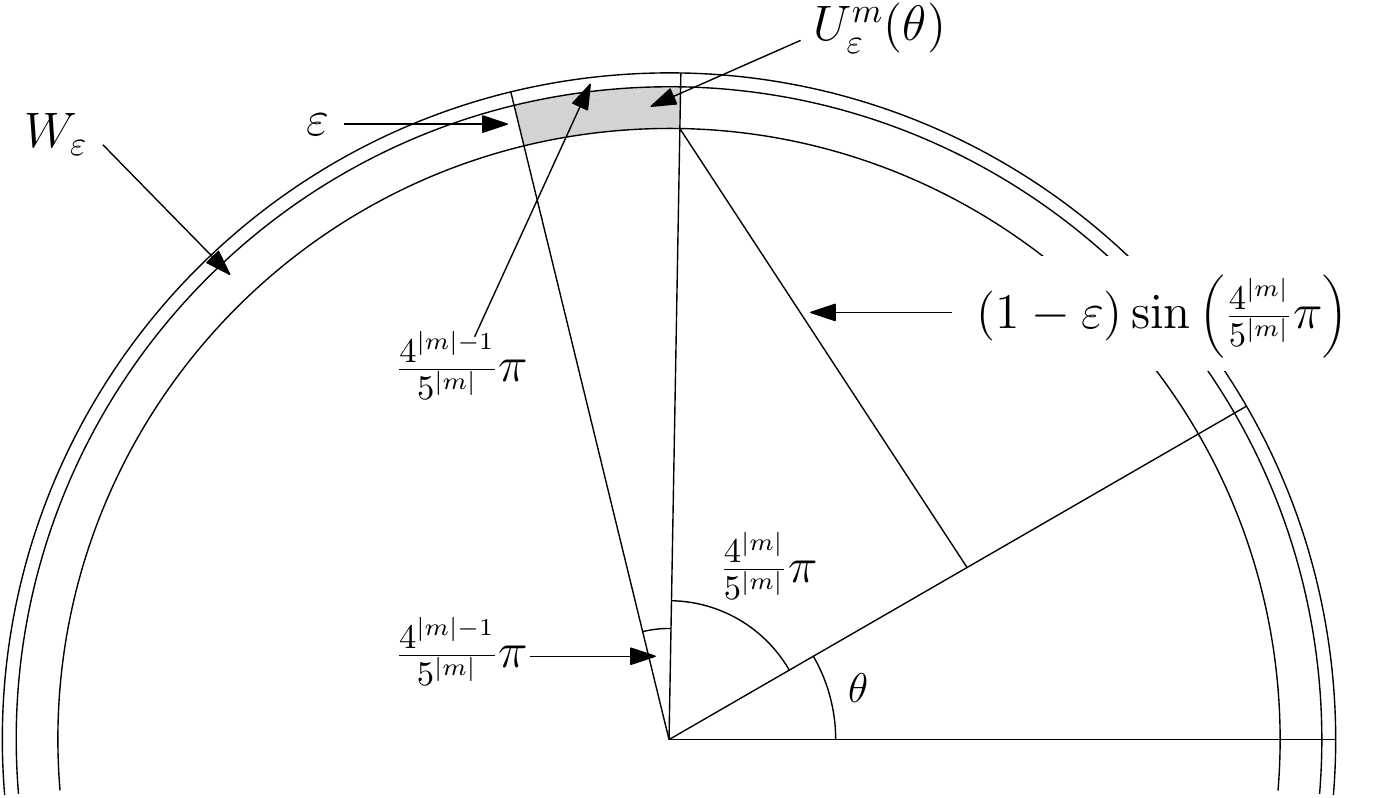}
\caption{Left: The decomposition of $W_\eps$ that it used to bound the second moment of $Z_\eps$. Each shaded region contributes at most a constant to the conditional first moment. Right: Illustration of the trigonometric calculations used to prove \cref{lem:trig}.}
\end{figure}

\begin{proof}
% By symmetry, it suffices to consider $\theta=0$.

% It is an exercise in elementary Euclidean geometry (see \cref{fig:trig}) to calculate that
% The diameter of $U^m_\eps$ is at most the length of the
We begin by calculating, using elementary trigonometry, that
% (see \cref{fig:trig})
\[\diam_\C(U^m_\eps(\theta)) \leq \frac{4^{|m|-1}}{5^{|m|}}\pi+\eps\]
and
\[d_\C(U^m_\eps(\theta), \{re^{i\theta}:r\geq0\}) = \left\{
\begin{array}{ll} 1-\eps & |m| \leq 3\\
(1-\eps)\sin\left(\frac{4^{|m|}}{5^{|m|}}\pi\right) & |m| \geq 4. \end{array} \right.\]
% It is then a calculus exercise to verify that the function
% $\sin(t)/\sqrt{1-\cos(t/4)}$
% is decreasing on the interval $(0,\pi/2]$. Thus,
By concavity of the function $\sin(t)$ on $[0,\pi]$, we have that $\sin(t)\geq 2t/\pi$ for all $t\in [0,\pi/2]$, and in particular
\begin{align*} \frac{\diam_\C(U^m_\eps(\theta))}{d_\C(U^m_\eps(\theta), \{re^{i\theta}:r\geq0\})}
% &\leq \sqrt{2-2\cos(\pi/8)} + \frac{\eps}{\sin(\frac{4^{|m|}}{5^{|m|}}\pi)}
 &\leq \frac{1}{8(1-\eps)}\pi +\frac{5^{|m|}\eps}{2\cdot 4^{|m|}(1-\eps)} \leq \frac{1}{8(1-\delta_2)}\pi + \frac{4^k\delta_4}{2\cdot 5^k (1-\delta_4)}
% {4^{|m|}} \sqrt{2-2\cos(\pi/8)} + \frac{2\cdot5^{|m|}\eps}{4^{|m|}}.
\end{align*}
% We deduce that
% \begin{align*} \frac{\diam_\C(U^m_\eps(\theta))}{d_\C(U^m_\eps(\theta), \{re^{i\theta}:r\geq0\})}
% % &\leq \sqrt{2-2\cos(\pi/8)} + \frac{\eps}{\sin(\frac{4^{|m|-1}}{5^{|m|-1}}\pi)}\\
%  &\leq (1-\eps)^{-1}\left(\sqrt{2-2\cos(\pi/8)} + \frac{2\cdot5^{-14}}{4^{-14}}\right)\end{align*}
 for all $\eps\leq\delta_4$ and $|m|\leq \log_{5/4}(1/\eps)-k$. This upper bound  is less than $\pi/7<1/2$ when $\delta_4$ is sufficiently small and $k$ is sufficiently large.
 % This concludes the proof.
  \qedhere

% \begin{figure}
% 	\includegraphics[width=0.7\textwidth]{trig4.pdf}
% 	\caption{Calculating $\diam_\C(U^m_\eps(\theta))$ and $d_\C(U^m_\eps(\theta),\{re^{i\theta} : r\geq0\})$.
% 	}
% 	\label{fig:trig}
% \end{figure}

\end{proof}
\medskip

\begin{lem} \label{lem:secondmoment} $\E[Z_\eps^2] \preceq \E[Z_\eps]\log(1/\eps)$ for all $0<\eps\leq \delta_4$, where $\delta_4$ is the constant from \cref{lem:trig}.
\end{lem}

Again, we remark that $\log(1/\eps)$ is of the same order as the hyperbolic distance between $e$ and $W_\eps$.
% \medskip

\begin{proof}
For each two vertices $u$ and $v$ in G, let
\[H(u,v):=\left\{ w \in V : |z(w)-z(u)| \leq |z(w)-z(v)|\right\}\]
be the set of vertices closer to $u$ than to $v$ with respect to the Euclidean metric on the circle packing.
% be the set of vertices $w$ of $G$ for which
Expand $\E[Z_\eps^2]$ as the sum
\[\E[Z_\eps^2] = \sum_{u,v \in W_\eps}\ah(u)\ah(v)\P\left(u,v \in \past_\F(e)\right).\]
If $u$ and $v$ are both in the past of $e$ in $\F$, let $w(u,v)$ be the first vertex at which the unique simple paths in $\F$ from $u$ to $e$ and from $v$ to $e$ meet. Then
% we have
% \begin{multline*}\E[Z_\eps^2] = 2\sum_{u,v \in W}\ah(u)r(v)^2\P\left( \begin{array}{l}u \in \past_\F(v), v \in \past_\F(v) \text{ and }\\ |z(w(u,v)) - z(u)| < |z(w(u,v)) - z(v)|\end{array}\right) \\
% + \sum_{u,v \in W}r(u)^2r(v)^2\P\left( \begin{array}{l}u \in \past_\F(v), v \in \past_\F(v) \text{ and }\\ |z(w(u,v)) - z(u)| = |z(w(u,v)) - z(v)|\end{array}\right) \end{multline*}
% and hence
\[ \E[Z_\eps^2] \leq 2\sum_{u,v \in W_\eps}\ah(u)\ah(v)\P\left(
% \begin{array}{l}
% \sP_u\cap\sP_v \cap \sF \cap
% \{
u,v \in \past_\F(e) \text{ and }
w(u,v) \in H(u,v)
% \}
% \\ |z(w(u,v)) - z(u)| \leq |z(w(u,v)) - z(v)|
% \end{array}
\right). \]
Consider generating $\F$ using Wilson's algorithm rooted at infinity, starting first with $u$ and then $v$. In order for $u$ and $v$ both to be in the past of $e$ and for $w(u,v)$ to be in $H(u,v)$, we must have that the simple random walk from $v$ hits $H(u,v)$,
so that
\begin{equation}\label{eq:Hestimate} \E[Z_\eps^2] \leq 2 \sum_{u\in  W_\eps} \ah(u)\P(u \in \past_\F(e))\sum_{v\in W_\eps}\ah(v)\bP_{v}(\tau_{H(u,v)}<\infty).\end{equation}
Thus, it suffices to prove that
 % and $C_5$
 % such that
% \begin{equation} \sum_{u\in  W_\eps} \bP_{u }(\tau_v<\infty) \leq C_4 \text{ for all $k\geq 1$} \end{equation}
% and
\begin{equation}\label{eq:secondmomentsufficient} \sum_{v \in  W_\eps} \ah(v)\bP_{v }(\tau_{H(u,v)}<\infty) \preceq \log(1/\eps)
% \text{ for all $\eps>0$ and $u \in  W_\eps$.}
\end{equation}
for all $\eps \leq \delta_4$ and $u \in  W_\eps$.
\medskip

Recall the definition of $U^m_\eps$ from \cref{lem:trig}.
Let $k$ be the universal constant from \cref{lem:trig}, let $\ell(\eps)=\lceil\log_{5/4}(\eps)-k\rceil$, and let $\theta=\arg(u)$. For each $m\in \Z$ with $1\leq |m| \leq \ell(\eps)$, let $S^m_\eps(\theta)$ be the set of vertices whose centres are contained in $U^m_\eps(\theta)$, and let
\[
	S^0_\eps=S^0_\eps(u):=\left\{v \in  W_\eps : \left|\arg \frac{z(v)}{z(u)}\right|
	\leq \frac{5^{k}\pi\eps}{4^{k}} \right\}.
\]
Then
\begin{align}\sum_{v \in  W_\eps} \ah(v)\bP_{v }(\tau_{H(u,v)}<\infty)
&= \sum_{m=-\ell(\eps)}^{\ell(\eps)}\sum_{v \in S^m_\eps} \ah(v)\bP_{v }(\tau_{H(u,v)}<\infty)\nonumber\\
 &\leq \sum_{m=-\ell(\eps)}^{\ell(\eps)}\sum_{v \in S^m_\eps} \ah(v)\bP_{v }(\tau_{H(u,S^m_\eps)}<\infty) \, ,\label{eq:secondmomentU} \end{align}
where for a set $A\subset V$ and a vertex $u$ we denote $H(u,A) = \cup_{v\in A}H(u,v)$.
% \[H(u,A):=  \]
%\[H(u,A):=\left\{ w \in V : |z(w)-z(u)| \leq |z(w)-z(v)| \text{ for some $\asaff{v}\in A$} \right\} =\bigcup_{v\in A}H(u,v).\]

Since $S^m_\eps$ is contained in $W_\eps$, it is $C$-short-lived for some $C=C(\mathbf{M})$ by \cref{lem:escape0}.
Thus, applying \cref{Lem:esacape2,lem:conddiam} 
% together with \eqref{eq:secondmomentU}
 yields that
\begin{equation}\label{eq:boundsum}\sum_{v\in S^m_\eps }\ah(v)\bP_{v}(\tau_{H(u,S^m_\eps)}<\infty)
\preceq \frac{\diam_\C(S^m_\eps)^2}{\min \big ( \diam(S^m_\eps), d_\C(S^m_\eps, H(u, S^m_\eps)) \big )^2 } \,\,  . \end{equation}
%\asaff{We have $\diam(S^m_\eps) \asymp d_\C(u,S^m_\eps)$ since both are of order $(4/5)^m$ and so by \eqref{e:comparedist} we get}
Our goal is to show that the last ratio is $\preceq 1$. If the minimum in the denominator is attained on the first term, then we are done. Otherwise, we bound
\begin{eqnarray*}  d_\C(S^m_\eps, H(u, S^m_\eps)) &=& \inf_{v_1,v_2 \in S^m_\eps} d_\C(v_1, H(u,v_2)) \nonumber \\&\geq& \inf_{v_1 \in S^m_\eps} d_\C(v_1, H(u,v_1)) - \sup_{v_1,v_2\in S^m_\eps} d_\C(v_1,v_2) \nonumber \\ &\geq& \frac{1}{2}d_\C\left(u,S^m_\eps\right) -  \diam_\C\left(S^m_\eps\right) \, .\end{eqnarray*}
%
%Note that for every vertex $u$ and set $A$,
%\begin{align}d_\C(A,H(u,A))
% % =
%% \inf_{v_1,v_2 \in A} d_\C(v_1,H(u,v_2))
%\geq \inf_{v_1,v_2 \in A} d_\C(v_2,H(u,v_2)) - d_\C(v_1,v_2)
% \geq \frac{1}{2}d_\C(u,A) - \diam_\C(A).  \label{e:comparedist} \end{align}
%
%
%\begin{align}
%\sum_{v\in S^m_\eps }\ah(v)\bP_{v}(\tau_{H(u,S^m_\eps)}<\infty)
%	&\preceq \frac{\diam_\C(S^m_\eps)^2}{\left(\frac{1}{2}d_\C\left(u,S^m_\eps\right) -  \diam_\C\left(S^m_\eps\right)\right)^2} \, .\nonumber
%\end{align}
Since $S^m_\eps \subset U^m_\eps$ have that $\diam_\C(S^m_\eps) \leq \diam_\C(U^m_\eps)$ and furthermore $d_\C\left(u,S^m_\eps\right) \geq  d_\C(U^m_\eps(\theta), \{re^{i\theta}: r \geq0\})$. Thus \cref{lem:trig} implies that for
$\eps \leq \delta_4$ and $1 \leq |m| \leq \ell(\eps)$ we have that
$$ d_\C(S^m_\eps, H(u, S^m_\eps)) \succeq \diam_\C(U^m_\eps) \, .$$
We get that the right hand side of \eqref{eq:boundsum} is $\preceq 1$. We put this in \eqref{eq:secondmomentU} and use the fact that $\sum_{v \in S_\eps^0}\ah(v)\preceq 1$ to get \eqref{eq:secondmomentsufficient} and conclude. \end{proof}

We now have everything we need in place to prove \cref{T:Wtailcp}.

\begin{proof}[Proof of \cref{T:Wtailcp}] Denote by $\sR^e_\eps$ the event that $\past_\F(e)$ contains a vertex whose center is within Euclidean distance $\eps$ of the unit circle, and let $\sD^e_R$ be the event that $\diam_\bbH(\past_\F(e))\geq R$. Then, letting $\eps(R)=1-\tanh(R/2)$, we have that $\sR^e_{\eps(R)/2}\subseteq \sD^e_R \subseteq \sR^e_{\eps(R)}$.  Thus, to prove \cref{T:Wtailcp}, it suffices to prove that \[\log(1/\eps)^{-1}\preceq_e \P(\sR^e_\eps) \preceq \log(1/\eps)^{-1}\]
for all $\eps\leq 1/2$.

As discussed at the beginning of \cref{Sec:mainproof}, the proof of \cref{T:planeWUSF}, and in particular of the estimate \eqref{eq:upperbound}, adapts immediately to yield the desired upper bound when the analysis there is carried out using the double circle packing $(P,P^\dagger)$ of $G$ rather than the circle packing of $T(G)$. Indeed,  we replace $P(f)$ with $P^\dagger(f)$ in the definition of $V_z(r,r')$ above \cref{lem:uniformdisjointannuli} and observe that part (2) of \cref{lem:uniformdisjointannuli} holds by the Ring Lemma \ref{T:Ring} since $G$ has bounded degrees. Furthermore, the proof of \cref{lem: flow} applies as written once we replace $P(f)$ with $P^\dagger(f)$. The proof of \cref{T:planeWUSF} now follows almost verbatim; we only note that the procedure of diverting the random path from $T$ to $G$ works seamlessly in the double circle packing setting and \eqref{eq:upperbound} is obtained. 

The remainder of this proof is devoted to proving the lower bound. Towards that aim, note that by the definition of $W_\eps$, the random variable $Z_\eps$ defined in \eqref{eq:defzeps} is positive if and only if $\sR^e_\eps$ occurs.
Let $\delta_4$ be the constant from \cref{lem:trig} and let $\delta_5=\delta_5(\bM,r_\bbH(x))$ be the constant from \cref{lem:firstmoment}.
The lower bound of \cref{T:Wtailcp} now follows from  \cref{lem:firstmoment,lem:secondmoment} together with the Cauchy-Schwartz inequality, which imply that
\begin{equation*} \P(Z_\eps >0) \geq \frac{\E[Z_\eps]^2}{\E[Z^2_\eps]} \succeq_e \frac{1}{\log(1/\eps)} \end{equation*}
for all $\eps \leq \min\{\delta_4,\delta_5\}$. Since $\P(Z_\eps>0)$ is an increasing function of $\eps$ and $\min\{\delta_4,\delta_5\}$ is an increasing function of $r_\bbH(x)$, it follows that
\begin{equation*} \P(Z_\eps >0) \geq \P\bigl(Z_{\min\{\delta_4,\delta_5\}\eps}>0\bigr) \succeq_e \frac{1}{\log(1/\eps)} \end{equation*}
for all $\eps\leq 1/2$.
\end{proof}

\begin{proof}[Proof of \cref{Thm:Wareacp}]
We continue to use the notation from the proof of \cref{T:Wtailcp}.
  We first prove the upper bound.
 Let $R \geq 1$. Applying Markov's inequality and \cref{lem:firstmomentupper} we obtain that
% For the second term,
\begin{align}
\P\bigg(\sum_{k=0}^{R^{1/2}}  Z_{s^{-3k}} \geq R \bigg)
& \leq \frac{1}{R}\E\bigg[\sum_{k=0}^{R^{1/2}}  Z_{s^{-3k}}\bigg] \preceq  R^{-1/2},\label{eq:areatri2}
\end{align}
and hence
\begin{align}\P\bigg(\area_\bbH(\past_\F(e)) \geq R\bigg) &\leq \P\bigg(\sum_{k=0}^{R^{1/2}}  Z_{s^{-3k}} \geq R \bigg) +  \P\bigg(\sum^\infty_{k=R^{1/2}}Z_{s^{-3k}} > 0\bigg) \preceq R^{-1/2}, \label{eq:areatri1}
\end{align}
where the second inequality follows from \cref{T:Wtailcp}  and \eqref{eq:areatri2}.

We now obtain a matching lower bound.
Let $\delta_4$ be the constant from \cref{lem:trig} and let $\delta_5=\delta_5(\bM,r_\bbH(x))$ be the constant from \cref{lem:firstmoment}, and let \[R_0=R_0(\bM,r_\bbH(x))= (4/9)\log_{1/s}^2(1/\min\{\delta_4,\delta_5\}),\] so that that $s^{-3R^{1/2}/2}$ is less than both $\delta_4$ and $\delta_5$ for all $R \geq R_0$. Let $R\geq R_0$ and let
\[W =  \bigcup_{k=\frac{1}{2}R^{1/2}}^{R^{1/2}} W_{s^{-3k}}
\qquad
\text{
 and let
 }
 \qquad
 Z=\sum_{k=\frac{1}{2}R^{1/2}}^{R^{1/2}}Z_{s^{-3k}}. \]

\noindent
The argument used to derive \eqref{eq:Hestimate} is imitated and yields that
\begin{equation}\label{eq:areaexpand} \E[Z^2] \leq 2 \sum_{u\in W}\ah(u)\P(u\in \past_\F(e))\sum_{v\in W}\ah(v)\bP_{v}(\tau_{H(u,v)}<\infty).\end{equation}

Let $u \in W$ and let $\theta=\arg(z(u))$. Let $\ell(s^{-3k})$ and let the sets $S^m_{s^{-3k}}(\theta)$ be defined as in the proof of \cref{lem:secondmoment}. Then,
\begin{align}
\sum_{v\in W}\ah(v)\bP_{v}(\tau_{H(u,v)}<\infty) \leq \sum_{k=\frac{1}{2}R^{1/2}}^{R^{1/2}} \sum_{m=- \ell(s^{-3k})}^{\ell(s^{-3k})}\left[\sum_{v\in S^m_\eps}\ah(v)\bP(\tau_{H(u,S^m_\eps)}<\infty)\right]. \label{eq:bracket}
\end{align}
The arguments used to bound the sum on the left-hand side of \eqref{eq:boundsum} in the proof of \cref{lem:secondmoment} also yield that the sums appearing in the square brackets on the right-hand side of \eqref{eq:bracket}  are bounded above by a constant depending only on $\bM$. It follows that
\begin{equation}
\label{eq:old}
\sum_{v\in W}\ah(v)\bP_{v}(\tau_{H(u,v)}<\infty) \preceq R\end{equation}
for all $u\in W$. Putting \eqref{eq:areaexpand} and \eqref{eq:old} together yields that $\E[Z^2] \preceq R\E[Z]$ as claimed.

Next, \cref{lem:firstmoment} implies that $\E[Z]\succeq_e R^{1/2}$, while \cref{T:Wtailcp} implies that $\P(Z >0) \preceq R^{-1/2}$. Thus, there exists a positive constant $C=C(\bM,r_\bbH(x))$ such that $\E[Z \mid Z>0]\geq CR$.
Applying the second moment estimate above, the Paley-Zigmund Inequality implies that
  \begin{equation}
  \P\left(Z \geq \frac{C}{2}R \;\middle|\;  Z > 0\right) \geq \frac{\E[Z \mid Z>0]^2}{4\E[Z^2 \mid Z>0]} = \frac{\E[Z]^2}{4\E[Z^2]\P(Z>0)}\succeq_e 1.\label{eq:areafinal}
  \end{equation}
Combining \eqref{eq:areafinal} with the lower bound of \cref{T:Wtailcp} yields that
\[\P\left(\area_\bbH(\past_\F(e)) \geq \frac{C}{2}R \right) \geq \P\left(Z \geq \frac{C}{2}R\right) \succeq_e R^{-1/2}\]
for all $R\geq R_0$. Since the probability on the left hand side is decreasing in $R$, we conclude that
\[\P(\area_\bbH(\past_\F(e)) \geq R ) \succeq_e R^{-1/2}\]
as claimed.
 \qedhere

\end{proof}

\begin{proof}[Proof of \cref{T:Ftailcp}]
Let $\F$ be a spanning forest of $G$ and let $e^\dagger$ be the edge of $G^\dagger$ dual to $e=(x,y)$. The past of $e^\dagger$ in the dual forest $\F^\dagger$ is contained in the region of the plane bounded by $e$ and $\Gamma_\F(x,y)$, so that \[\diam_\bbH(\Gamma_\F(x,y)) \geq \diam_\bbH(\past_{\F^\dagger}(e^\dagger)).\]
Meanwhile, if $\past_{\F^\dagger}(e^\dagger)$ is non-empty, then every edge in the path $\Gamma_\F(x,y)$ is incident to a face of $G$ that is in $\past_{\F^\dagger}(e^\dagger)$. By the Ring Lemma (\cref{Thm:DCPRing}), the hyperbolic radii of circles in $P \cup P^\dagger$ are bounded above. We deduce that there exists a constant $C=C(\bM)$ such that
\[\diam_\bbH(\Gamma_\F(x,y))\leq \diam_\bbH(\past_{\F^\dagger}(e^\dagger))+C.\]
We deduce \cref{T:Ftailcp} from \cref{T:Wtailcp} by applying these estimates when $\F$ is the FUSF of $G$ and $\F^\dagger$ is the WUSF of $G^\dagger$.
\end{proof}

\subsection{The uniformly transient case}\label{Subsec:dualexponents}

Recall that a graph is said to be \textbf{uniformly transient} if $\mathbf{p} = \inf_{v \in V} \bP_v(\tau_v^+=\infty)$ is positive. If in addition the graph has bounded degrees, this is equivalent to the property that $\Ceff(v\leftrightarrow \infty)$ is bounded away from zero uniformly in $v$.

\begin{prop}\label{Prop:radii}
Then there exists a constant $C=C(\bM)$ such that
\[  r_{\bbH}(v) \geq {1 \over C} \exp\left(-C \Reff(v \leftrightarrow \infty)\right)  . \]
\end{prop}

\begin{proof}
By applying a M\"obius transformation if necessary, we may assume that the circle  $\Pp{v}$ is centered at the origin. By the Ring Lemma (\cref{Thm:DCPRing}), $r_{\bbH}(v)$ is bounded above by a constant depending only on the maximum degree and codegree of $G$. Applying \cite[Corollary 3.3]{GGN13} together with \cref{Thm:DCPRing} yields that $\Reff(v \leftrightarrow \infty) \geq c \log(1/r(v))$ for some constant $c=c(\bM)$. Since the hyperbolic radii are bounded from above, the Euclidean radius $r(v)$ is comparable to $r_{\bbH}(v)$.
\end{proof}

We are now ready to prove \cref{cor:Flengthtail}. In the rest of this subsection, we will use $\preceq,\succeq$ and $\asymp$ to denote inequalities or equalities that hold up to positive multiplicative constants depending only on $\bM$ and $\mathbf{p}$.

\begin{proof}[Proof of \cref{cor:Flengthtail}] Let $e=(x,y)$ be an edge of $G$ and let $\F$ be sample of $\FUSF_G$. By \cref{thm:duality} the dual forest $\F^\dagger$ is distributed as $\WUSF_G$ and so by \cref{T:planeWUSF} it is almost surely one-ended. Hence $\past_{\F^\dagger}(e^\dagger)$ (recall the definition above \cref{T:Wtailcp}), where $e^\dagger$ the edge of $G^\dagger$ dual to $e$, is well defined.

The past of $e^\dagger$ in the dual forest $\F^\dagger$ is contained in the region of the plane bounded by $e$ and $\Gamma_\F(e)$. The non-amenability of the hyperbolic plane implies that the perimeter of any set is at least a constant multiple of its area, and so
\begin{equation}\label{eq:length1} \sum_{v\in \Gamma_\F(x,y)}r_\bbH(v) \succeq \area_\bbH\left(\past_{\F^\dagger}(e^\dagger)\right).\end{equation}
 On the other hand, if $\past_{\F^\dagger}(e^\dagger)$ is non-empty, then every edge in the path $\Gamma_\F(x,y)$ is incident to a face of $G$ that is in $\past_{\F^\dagger}(e^\dagger)$.
 We deduce that if $\past_{\F^\dagger}(e^\dagger)$ is non-empty then
\begin{equation}\label{eq:length2}\area_\bbH\left(\past_{\F^\dagger}(e^\dagger)\right) \succeq \sum_{v\in \Gamma_\F(x,y)}a_\bbH(v).\end{equation}
Note that neither estimate \eqref{eq:length1} or \eqref{eq:length2} required uniform transience.
\cref{Prop:radii} and the Ring Lemma (\cref{Thm:DCPRing}) imply that
\[ |\Gamma_\F(x,y)| \asymp \sum_{v\in \Gamma_\F(x,y)}r_\bbH(v) \asymp \sum_{v\in \Gamma_\F(x,y)}a_\bbH(v). \]
Thus, we deduce \cref{cor:Flengthtail} from \cref{Thm:Wareacp}.
\end{proof}

Let $(X_1,d_1)$ and $(X_2,d_2)$ be metric spaces and let $\alpha,\beta$ be positive. A (not necessarily continuous) function $\phi:X_1\to X_2$ is said to be an \textbf{$(\alpha,\beta)$-rough isometry} if the following hold.
\begin{enumerate}[leftmargin=*]
\item (\textbf{$\phi$ roughly preserves distances.}) $\alpha^{-1}d_1(x,y) -\beta \leq d_2(\phi(x),\phi(y)) \leq \alpha d_1(x,y)+\beta$ for all $x,y \in X_1$.
 \item (\textbf{$\phi$ is almost surjective.}) For every $x_2 \in X_2$, there exists $x_1 \in X_1$ such that $d_2(\phi(x_1),x_2)\leq \beta$.
\end{enumerate}
See \cite[\S 2.6]{LP:book} for further background on rough isometries. We write $d_G$ for the graph distance on $V$.

\begin{corollary}\label{cor:rough}
Let $G$ be a uniformly transient, polyhedral, proper plane network with bounded codegrees and bounded local geometry. Let $d_G$ denote the graph distance on $V$, let $(P,P^\dagger)$ be a double circle packing of $G$ in $\D$, and let $z(v)$ be the centre of the disc in $P$ corresponding to the vertex $v$. Then there exist positive constants $\alpha=\alpha(\bM,\mathbf{p})$ and $\beta=\beta(\bM,\mathbf{p})$ such that  $z$ is an $(\alpha,\beta)$-rough isometry from $(V,d_{G})$ to $(\D,d_{\bbH})$.
\end{corollary}

\begin{proof} \cref{Prop:radii} implies that for every vertex $v$ of $G$, $r_{\bbH}(v)$ is bounded both above and away from zero by positive constants. Almost surjectivity is immediate. For each two vertices $u$ and $v$ in $G$, the shortest graph distance path between them induces a curve in $\D$ (by going along the hyperbolic geodesics between the centres of the circles in the path) whose hyperbolic length is  $ \succeq d_G(u,v)$.

Conversely, let $\gamma$ be the hyperbolic geodesic between $z(u)$ and $z(v)$,  and consider the set $W$ of vertices $w$ of $G$ such that either $\Pp{w}$ intersects $\gamma$ or $\Pd{f}$ intersects $\gamma$ for some face $f$ incident to $w$. Let $d$ be the length of $\gamma$.
 Since all circles in $(P,P^\dagger)$ have a uniform upper bound on their hyperbolic radii, we deduce that all circles in $W$ are contained in a hyperbolic neighbourhood of constant thickness about $\gamma$, and hence the total area of these circles is $\preceq d$. Since the radii of the circles are also bounded away from zero, we deduce that the cardinality of $W$ is also $\preceq d$. Since $W$ contains a path in $G$ from $u$ to $v$, we deduce that $d_G(u,v)$ is $\preceq d$ as required.
\end{proof}

We summarise the situation for uniformly transient graphs in the following corollary, which follows immediately by combining \cref{cor:rough} with \cref{T:Ftailcp,T:Wtailcp,Thm:Wareacp,cor:Flengthtail}. We write $\diam_G$ for the graph distance diameter of a set of vertices in~$G$.

\begin{corollary}[Graph distance exponents]\label{Cor:graphexponents}
Let $G$ be a uniformly transient, polyhedral, proper plane network with bounded codegrees and bounded local geometry, and let $\mathbf{p}>0$ be a uniform lower bound on the escape probabilities of $G$. Then there exist positive constants $k_1=k_1(\bM,\mathbf{p})$ and $k_2=k_2(\bM,\mathbf{p})$ such that
\[
\hspace{1cm}
\begin{array}{llcll}
% \begin{align*}
k_1 R^{-1} &\leq &\FUSF_G(\diam_G(\Gamma_\F(x,y)) \geq R) &\leq &k_2 R^{-1},\vspace{.25em}\\
% \item
k_1 R^{-1} &\leq &\WUSF_G(\diam_G(\past_\F(e))\geq R) &\leq &k_2 R^{-1},\vspace{.25em}\\
k_1 R^{-1/2} &\leq &\WUSF_G(|\past_\F(e)|\geq R) &\leq &k_2 R^{-1/2}, %\text{ and } \vspace{.25em}
\\
k_1 R^{-1/2} &\leq &\FUSF_G(|\Gamma_\F(x,y)|\geq R) &\leq &k_2 R^{-1/2}
\vspace{.25em}
\end{array}
\]
for every edge $e=(x,y)$ of $G$ and every $R\geq 1$, where $\F$ is a sample of either the free or wired uniform spanning forest of $G$ as appropriate.
\end{corollary}

\section{A remark and a problem}\label{sec:conclusion}
\subsection{A remark}
\begin{remark}[Non-universality in the parabolic case.]\label{remark:parabolic}

 Unlike in the CP hyperbolic case, the exponents governing the behaviour of the USTs of CP parabolic, polyhedral proper plane graphs with bounded codegrees and bounded local geometry are not universal, and need not exist in general.

Indeed,
consider the double circle packing of the proper plane quadrangulation with underlying graph $\N\times \Z_4$, pictured in \cref{fig:dcptube}.
Let the packing be normalised to be symmetric under rotation by $\pi/2$ about the origin and to have $r(0,0)=1$. It is possible to compute that
$r(i,j)=(3+2\sqrt{2})^{i}$ and hence that $|z(i,j)|$ is comparable to $(3+2\sqrt{2})^{i}$ for every $(i,j)\in \N\times \Z_4$.
 Suppose that the edges connecting $(i,j)$ to $(i\pm 1,j)$ are given weight $1$ for every $(i,j)\in \N\times\Z_4$, while the edges connecting $(i,j)$ to $(i,j\pm 1)$ are given weight $c$ for each $(i,j)\in \N\times \Z_4$.
 It can be computed that the probability that a walk started at $(i,0)$ hits $(0,0)$ without ever changing its second coordinate is
$a(c)^i:=(1+c-\sqrt{c^2+2c})^i$.
Let $e=((0,0),(0,1))$. By running Wilson's algorithm rooted at $(0,0)$ starting from the vertices $(i,0)$ and $(i,1)$, we see that
\begin{align*}\UST\left(\past_T(e) \cap \{i\}\times \Z_4 \neq \emptyset\right) &\geq \bP_{(i,0)}(\tau_{(0,0)}<\tau_{\N \times\{1,2,3\}})\bP_{(i,1)}(\tau_{(0,1)}<\tau_{\N \times\{0,2,3\}})\\ & \hspace{6cm}\cdot\bP_{(0,1)}(X_1=(0,0))
\\
&=\frac{c}{2c+1}a(c)^{2i}. \end{align*}
the right-hand side is exactly the probability that the random walk from $(i,0)$ hits $(0,0)$ without ever changing its second coordinate, and that the random walk from $(i,1)$ hits $(0,1)$ without ever changing its second coordinate and then steps to $(0,0)$.
\begin{figure}
\centering
\includegraphics[width=0.375\textwidth]{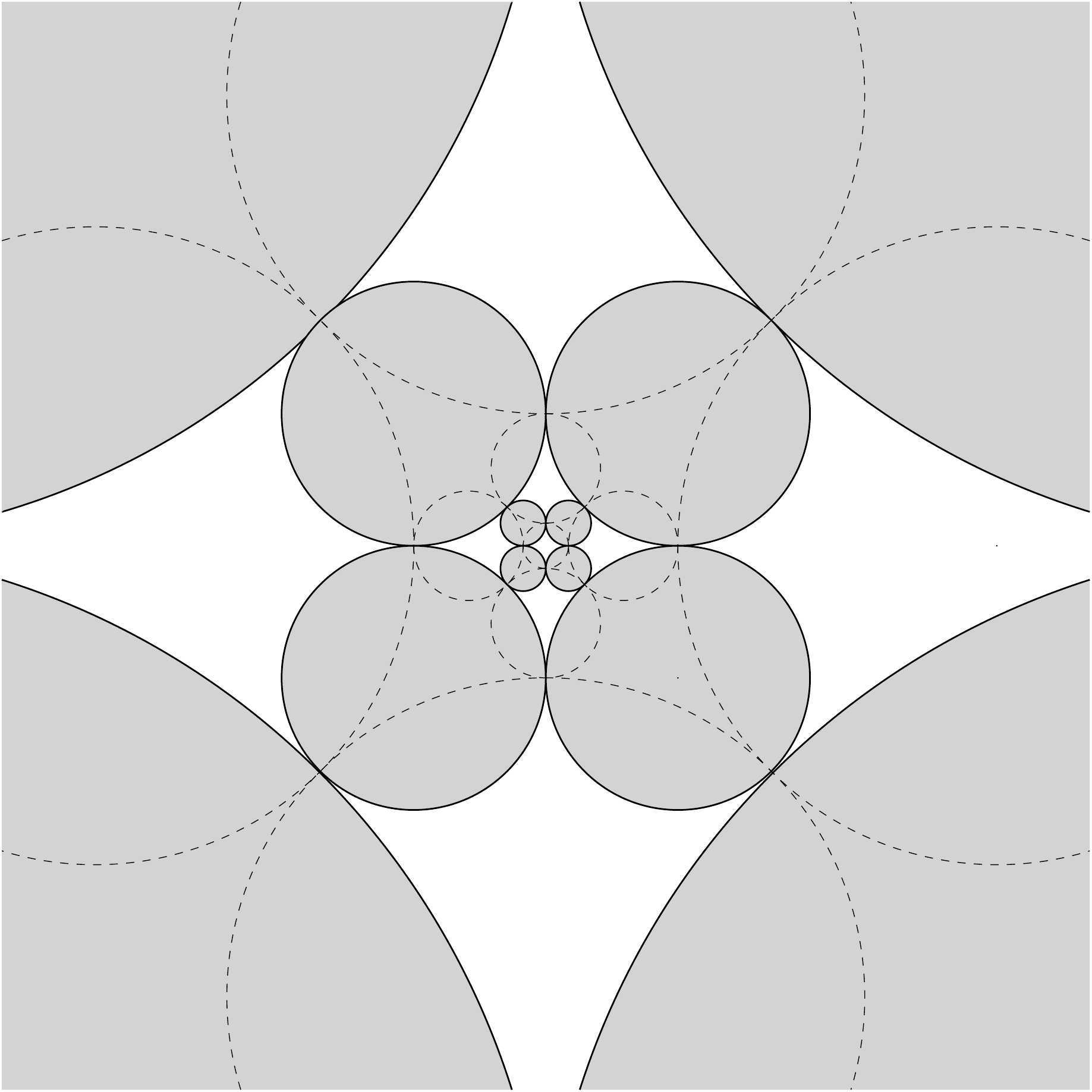}
\caption{The double circle packing of $\N \times \Z_4$.}
\label{fig:dcptube}
\end{figure}

Let $q(c)$ be the probability that a random walk started at $(i,j)$ visits every vertex of $\{i\}\times \Z_4$ before changing its vertical coordinate, which tends to one as $c\to\infty$.  Let $Y$ be a loop-erased random walk from $(0,0)$. It can be computed that the probability that a random walk  started from $(i,j)$ visits $\{0\}\times \Z_4$ before hitting the trace of $Y$ is at most $b(c)^i:=(1-q(c))^i$. Thus, by Wilson's algorithm and a union bound,
\[\UST(\past_T(e)\cap \{i\}\times \Z_4 \neq \emptyset) \leq 4b(c)^i.\]
It follows that there exist positive constants $k(c)$, $\alpha(c)$ and $\beta(c)$ such that $\alpha(c)\to0$ as $c\to 0$,  $\beta(c)\to\infty$ as $c\to\infty$, and
\[k(c)^{-1}R^{-\alpha(c)}\leq \UST\left(\mathrm{diam}_\C(\past_T(e)) \geq R\right) \leq k(c)R^{-\beta(c)}. \]

Thus, by varying $c$, we obtain CP parabolic proper plane networks with bounded codegrees and bounded local geometry with different exponents governing the diameter of the pasts of edges in their USTs: If $c$ is large the diameter has a light tail, while if $c$ is small the diameter has a heavy tail.
Furthermore, by varying the weight of $((i,j),(i,j\pm1))$ as a function of $i$ in the above example (i.e., making $c$ small at some scales and large at others), it is possible to construct a polyhedral, CP parabolic proper plane network $G$ with bounded codegrees and bounded local geometry such that
\[\frac{\log\UST_G\left(\mathrm{diam}_\C(\past_T(e)) \geq R\right)}{\log(R)}\]
does not converge as $R\to\infty$ for some edge $e$ of $G$. The details are left to the reader.

Similar constructions show that the behaviour of
$\WUSF_G(\diam_\bbH(\past_\F(e))\geq R\cdot r_\bbH(x))$
is not universal over polyhedral, CP hyperbolic proper plane network $G$ with bounded codegrees and bounded local geometry in the regime that $r_\bbH(x)$ is small.

\end{remark}

\subsection{A problem}

It is natural to ask to what extent  the assumption of planarity in \cref{T:mainthm} can be relaxed.
Part (1) of the following question was suggested by R.\ Lyons.

\begin{question}\label{Q:product}
Let $G$ be a bounded degree proper plane graph.
\begin{enumerate}
% [leftmargin=*]
\item Let $H$ be a finite graph.
 Is the free uniform spanning forest of the product graph $G\times H$ connected almost surely?
 \item Let $G'$ be a bounded degree graph that is rough isometric to $G$. Is the
 the free uniform spanning forest of $G$ connected almost surely?
 \end{enumerate}
\end{question}

Without the assumption of planarity, connectivity of the FUSF is not preserved by rough isometries; this can be seen from an analysis of the graphs appearing in \cite[Theorem 3.5]{BS96a}.

\subsection*{Acknowledgments}
TH thanks Tel Aviv University and both authors thank the Issac Newton Institute, where part of this work was carried out, for their hospitality. We also thank Tyler Helmuth for finding several typos in an earlier version of this manuscript. Images of circle packings were created with Ken Stephenson's CirclePack software \cite{CP}.
This project was supported by NSERC Discovery grant 418203, ISF grant 1207/15, and ERC starting grant 676970.

\footnotesize{
\bibliographystyle{abbrv}
 \bibliography{unimodular}
 }
 \end{document}